\documentclass[12pt]{amsart}
\usepackage[utf8]{inputenc} \usepackage[T1]{fontenc}
\usepackage{lmodern}
\usepackage[margin=2.5cm, centering]{geometry}

\usepackage[english]{babel} 
\usepackage{geometry}
\usepackage{caption}
\usepackage{subcaption}
\usepackage{mathrsfs,amsmath,amsfonts,amstext,amssymb,amsthm,amsopn,amsthm,color, dsfont,bm,bbm}
\usepackage{enumitem}
\usepackage[titletoc]{appendix}
\usepackage[dvipsnames]{xcolor}
\usepackage{stackrel}
\usepackage{bigints}

\usepackage{thmtools}
\usepackage[colorlinks=true, linkcolor=blue, citecolor=black]{hyperref}
\addto\extrasenglish{}
\addto\extrasenglish{}

\numberwithin{equation}{section}
\declaretheorem[name=Theorem, numberwithin=section]{theorem}
\newtheorem{lemma}[theorem]{Lemma}
\newtheorem{proposition}[theorem]{Proposition}
\newtheorem{corollary}[theorem]{Corollary}

\declaretheoremstyle[bodyfont=\normalfont]{remark-style}
\declaretheorem[name={Remark}, style=remark-style, sibling=theorem]{remark}

\newcommand{\R}{\mathbb{R}}

\newcommand{\C}{\mathbb{C}}

\newcommand{\E}{\mathbb{E}}
\renewcommand{\P}{\mathbb{P}}

\newcommand{\BB}{\mathcal{B}}

\newcommand{\ind}{{\bf 1}}
\renewcommand{\leq}{\leqslant} 
\renewcommand{\le}{\leq}
\renewcommand{\geq}{\geqslant} 
\renewcommand{\ge}{\geq}
\renewcommand{\Re}{\operatorname{Re}}
\renewcommand{\Im}{\operatorname{Im}}

\DeclareMathOperator{\supp}{supp}

\renewcommand{\epsilon}{\varepsilon}
\def\Go{G_{\{0\}^c}}
\newcommand{\tauD}{\tau_D}
\newcommand{\tauoR}{\tau_{(0,R)}}
\newcommand{\tauRR}{\tau_{(-R,R)}}
\newcommand{\tauoinf}{\tau_{(0,\infty)}}
\newcommand{\WLSC}[2]{{\rm WLSC}(#1,#2)}

\begin{document}
	\title{Hitting probabilities for L\'{e}vy processes on the real line}
	\author{Tomasz Grzywny}
	\address{
	Tomasz Grzywny\\
	Wydzia{\l} Matematyki,
	Politechnika Wroc{\l}awska\\
	Wyb. Wyspia\'{n}skiego 27\\
	50-370 Wroc{\l}aw\\
	Poland}
	\email{tomasz.grzywny@pwr.edu.pl}
	
	\author{\L{}ukasz Le\.{z}aj}
	\address{
		\L{}ukasz Le\.{z}aj\\
		Wydzia{\l} Matematyki,
		Politechnika Wroc{\l}awska\\
		Wyb. Wyspia\'{n}skiego 27\\
		50-370 Wroc{\l}aw\\
		Poland}
	\email{lukasz.lezaj@pwr.edu.pl}
	
	\author{Maciej Mi\'{s}ta}
	\address{
		Maciej Mi\'{s}ta\\
		Wydzia{\l} Matematyki,
		Politechnika Wroc{\l}awska\\
		Wyb. Wyspia\'{n}skiego 27\\
		50-370 Wroc{\l}aw\\
		Poland}
	\email{maciej.mista@pwr.edu.pl}
		
	\date{}
	\thanks{The research was partially supported by  National Science Centre (Poland): grant 2015/17/B/ST1/01043}
	\begin{abstract}
		We prove  sharp two-sided estimates on the tail probability of the first hitting time of bounded interval as well as its asymptotic behaviour for general non-symmetric processes which satisfy an integral condition
		\[
		\int_0^{\infty} \frac{d\xi}{1+\Re \psi(\xi)}<\infty.
		\]
		To this end, we first prove and then apply the global scale invariant Harnack inequality. Results are obtained under certain conditions on the characteristic exponent. We provide a wide class of L\'{e}vy processs which satisfy these assumptions.
	\end{abstract}
	\maketitle
	
\section{Introduction}
The aim of this paper is to discuss the distribution of the first hitting time of the point or the bounded interval for non-symmetric L\'{e}vy processes which satisfy the following condition:
\[
\int_0^{\infty} \frac{d\xi}{1+\Re \psi(\xi)} < \infty.
\]
Such condition implies that $0$ is regular for itself. Under some regularity assumptions we prove sharp two-sided estimates on the tail probability of the first hitting time of a point or a bounded interval, as well as its asymptotic behaviour. In the case of symmetric processes, it can be described by the compensated potential kernel (see \cite{GR17}), which is given by
\[
K(x) = \int_0^{\infty} \big( p(s,0)-p(s,x) \big)\,ds,
\] 
where $p(s,\cdot)$ is the transition density of the process (which exists due to the integral condition). However, in our setting one of substantial difficulties one has to overcome is the fact that we do not a priori know if $K$ exists, and even if it does exist, it may vanish on the whole half-line. For symmetric L\'{e}vy processes its existence is an easy consequence of the monotone convergence theorem, but in our case we are, in general, forced to adopt a different method. Instead, we propose an approach based on the symmetrized compensated potential kernel
\[
H(x) = \frac{1}{\pi} \int_0^{\infty} (1-\cos xs) \Re \Bigg[ \frac{1}{\psi(s)} \Bigg]\,ds,
\]
which turns out to be the proper object for description of the behaviour of the first hitting time. Its huge advantage is the fact that the integrability condition we assume in the whole paper ensures that $H$ is well-defined and therefore can serve our purpose. 

Let us briefly describe our results. First, we concentrate on the asymptotic behaviour of the first hitting time of arbitrary compact sets which contain the origin (Theorem \ref{gtw} and Corollary \ref{wgtw}). The obtained asymptotics hold true if $\Re \psi$ varies regularly with parameter $\alpha \in (1,2]$ and $\Im \psi$ displays a similar behaviour. We devote Section \ref{sec:regvar} to discussion about situations in which such condition holds true (see Theorem \ref{thm:Re_asymp_Im}). In particular, that turns out to be true if the L\'{e}vy measure is of the special form
\[
\nu(dx) = C_d \ind_{x<0}\nu_0(dx)+C_u \ind_{x>0}\nu_0(dx).
\]
An important class of processes which clearly exhibit such behaviour are spectrally one-sided L\'{e}vy processes.

Next, we turn our attention to derivation of sharp two-sided estimates on the first hitting time. To that end, we prove the global scale invariant Harnack inequality under weak lower scaling condition on the real part of the characteristic exponent. That result seems to be of some value in and of itself, as Harnack inequality is very strong theoretic tool in the potential theory. For instance, its easy consequences are estimates on the derivative of the renewal function for the ladder height process, which in turn entail estimates on the density of the distribution of the supremum process (see \cite{MR3531705}). We also provide estimates on the boundary behaviour of harmonic functions. Meanwhile, we derive estimates on the Green function of a bounded interval and a half-line (Lemma \ref{lem:1} and Corollary \ref{cor:green_0inf}), which also seem to be interesting on its own. For example, the first one together with some regularity assumptions on the L\'{e}vy masure imply the boundary Harnack inequality, which is also an important potential theoretic tool (see e.g. \cite{MR3271268}).

Eventually, we apply those results to obtain estimates on the tail of the first hitting time of points and intervals. They are derived under assumptions of global lower scaling property, zero mean and control of $K^{\lambda}$ from below by $H$ for small $\lambda$. The estimates are expressed in terms of symmetrized compensated potential kernel $H$ and renewal function for the dual process $\widehat{V}$, and therefore, in particular, do not require the existence of the compensated potential kernel $K$. Nevertheless, in Section \ref{sec:asymptotics} we provide conditions which assure that $K$ is well-defined. We remark here that if the process is symmetric, then our assumptions reduce to those obtained in \cite{GR17}. Since the third assumption is not a priori obvious for general non-symmetric processes, in Subsection \ref{subsec:ex} we present an example of wide class of processes for which such property holds true. Furthermore, if we restrict ourselves to the special case of spectrally negative L\'{e}vy processes then due to its specific structure, we are able to prove sharp two-sided estimates on both sides of the interval and for any $t>0$ (see Corollary \ref{cor:est_spec_neg}). To our best knowledge, such results have not been obtained before. Apart from its own value, they can be used, together with heat kernel estimates (\cite{GS2019}) for instance for estimation of the Hausdorff dimension of the inverse images of L\'{e}vy processes (see \cite{Park2019}). We also remark that although our main object to operate with is the real part of the characteristic exponent, one can work with the tail of the L\'{e}vy measure instead, since in view of \cite[Proposition 3.8]{GLT18}, scaling property of the latter implies scaling of the former.

The first studies on the first hitting time of a point or a compact set concerned $\alpha$-stable processes. The asymptotic behaviour in the case of recurrent $\alpha$-stable process, i.e. $1<\alpha \leq 2$, for arbitrary compact sets, was  derived in \cite{Port67}. Next, in \cite{YYY79} the authors discuss the law of the first hitting time of a point for the symmetric $\alpha$-stable processes with $1<\alpha\leq 2$. A series representation of the density of the first hitting time of a point in the case of spectrally positive $\alpha$-stable L\'{e}vy processes, $1<\alpha<2$, was obtained in \cite{Peskir08} and \cite{Simon11}. That result was extended two general spectrally two-sided $\alpha$-stable processes with $1<\alpha<2$ in \cite{KKPW14}. Let us note here that in case of spectrally negative L\'{e}vy processes starting from the left side of the interval, the first hitting time is equal to the first passage time through the left end and in consequence, one may apply tools from the fluctuation theory to handle the problem (see e.g. \cite{Bertoin96} or \cite{Sato99} for details).

The general symmetric case is much harder to handle and in principle, requires some regularity assumptions on the characteristic exponent of the process. In \cite{MK12spectral} the author derives an integral representation of the distribution function of the first hitting time of a point in terms of eigenfunctions of the semigroup of the process killed upon hitting the origin. That idea was later adopted in \cite{KJ15} to obtain the asymptotic expansion of the distribution function (and its derivatives) of the first hitting time of a point for symmetric L\'{e}vy processes with completely monotone jumps. Recently, in \cite{JM19}, the ideas from \cite{MK12spectral} were extended to non-symmetric L\'{e}vy processes. A different approach was proposed in \cite{GR17}, where the authors prove and apply the global Harnack inequality in order to obtain sharp estimates on the tail probability  of the first hitting time of points and bounded intervals for symmetric processes under global lower scaling condition imposed on the characteristic exponent. In the present paper we generalize these ideas to non-symmetric L\'{e}vy processes,  which, to our best knowledge, has not been investigated in such generality before.

The article is organized as follows. In Section \ref{sec:prelims} we introduce our setting, basic objects and tools exploited in the paper and prove some auxiliary results. In Section \ref{sec:regvar} we prove that some specific form of the L\'{e}vy measure and regular variation of $\Re \psi$ implies that $\Im \psi$ is in fact, comparable to the real part; we apply that result in Section \ref{sec:asymptotics}, where asymptotic behaviour of the tail of the distribution of the first hitting time is obtained. Here, Theorem \ref{thm:Re_asymp_Im} provides an important example. Section \ref{sec:harmonic} is devoted to the proof of the Harnack inequality, some of its consequences and boundary behaviour of harmonic functions. We apply those results in Section \ref{sec:estimates}, where we provide sharp two-sided estimates on the tail probability of the first hitting time of a bounded interval. Finally, in Subsection \ref{subsec:ex} we indicate a wide class of non-symmetric processes which satisfy assumptions of Theorem \ref{thm:main_hitting_time}.
%Known results: to co we wstępie do pracy o symetrycznych + jej wyniki oraz ostatnia praca Jacka Muchy
%
%Our main results and applications:
%\begin{itemize}
%\item equivalence between regular variation of $\Re\psi$ at infinity with regular variation of the symmetrized tail of the L\'{e}vy measure 
%\item Harnack inequality = very strong potential theoretic tool. Easy consequence is estimates of the derivative of the renewal function for the ladder height processes. This implies the estimates on the density of the distribution of supremum process see \cite{MR3531705}
%\item Estimates of Green function of a half line and interval. The last one + some regularity assumption on the L\'{e}vy measure imply the boundary Harnack inequality = important in the potential theory \cite{MR3271268}
%\item Estimates on the tail of hitting time to points and intervals. Applications: this + heat kernel estimates [TG+KS 2019] can be used for instance to estimate the Hausdorff dimension see \cite{Park2019}  
%\end{itemize}
\section{Preliminaries}\label{sec:prelims}
\noindent {\bf Notation}. Throughout the paper $c,c_1,C_1,...$ denote positive constants which may vary from line to line during estimates. We write $c=c(a)$ when $c$ depends only on parameter $a$. For two numbers $a$, $b$ we denote $a \wedge b = \min \{a,b\}$. For two functions $f,g$ we write $f \approx g$ if the ratio $f(u)/g(u)$ stays between two positive constants. Similarly, we write $f \lesssim g$ ($f \gtrsim g$) if the ratio is bounded from above (below) by a positive constant. By $f \cong cg$, $x \to x_0$, we mean that $\lim_{x \to x_0} f(x)/g(x)=c$. We write $f(x) = \mathcal{O}(g(x))$, $x \to x_0$, if $\limsup_{x \to x_0} |f(x)|/g(x)<\infty$. For a complex-valued function $f \colon \R\mapsto \C$, by $f^{-1}$ we denote its generalized inverse, that is $f^{-1}(s) = \sup \{r>0\colon f^*(r)=s\}$, where $f^*(r) = \sup_{|x| \leq r} \Re f(x)$. Borel sets on the real line are denoted by $\BB_{\R}$ .

Throughout the whole paper we let $\mathbf{X}=(X_t\colon t \geq 0)$ be a one-dimensional L\'{e}vy process with the L\'{e}vy measure $\nu$. Recall that any L\'{e}vy measure satisfies the following condition:
\[
\int_{\R} \big(1 \wedge x^2\big)\nu({dx})<\infty.
\]
If we let $\psi$ be its characteristic exponent, that is
\[
\E e^{-\xi X_t} = e^{-t\psi(\xi)}, \quad \xi \in \R,
\] 
then it is well known that $\psi$ is of the form
\[
\psi(\xi) = \sigma^2 \xi^2-i\gamma \xi - \int_{\R} \big( e^{i\xi z}-1-i\xi z\ind_{|z|<1} \big)\,\nu(dz), \quad \xi \in \R,
\]
where $\sigma >0$ and $\gamma \in \R$. Note that since we do not assume symmetry of $X_t$, both $\nu$ and $\psi$ need not be symmetric. The triplet $(\sigma, \gamma, \nu)$ uniquely determines the L\'{e}vy process and is therefore called \emph{the generating triplet} of $\mathbf{X}$. Let us notice that
\[
\Re \psi(\xi) = \sigma^2 \xi^2+ \int_{\R} \big( 1-\cos \xi z \big)\,\nu(dz), \quad \xi \in \R,
\]
and
\[
\Im \psi (\xi) = -\gamma\xi + \int_{\R} \big( \xi z \ind_{|z|<1} -\sin \xi z \big)\,\nu(dz), \quad \xi \in \R.
\]
Observe that $\Re \psi$ is symmetric even if $\mathbf{X}$ is not symmetric. If we assume that $\int_{(-1,1)^c} |z|\,\nu(dz)<\infty$, then we can also write
\begin{equation}\label{eq:Im_reprezentation}
\Im \psi(\xi) = -\gamma_1 \xi +\int_{\R} \big( \xi z-\sin \xi z \big)\,\nu(dz),
\end{equation}
where $\gamma_1=\gamma+\int_{(-1,1)^c}z\,\nu(dz)$. In particular, if $\E X_1=0$ then $\gamma_1=0$.

Our standing assumption is finiteness of the following integral:
\begin{equation}\label{eq:main_assumption}
\int_0^{\infty} \frac{d\xi}{1+\Re \psi(\xi)}<\infty.
\end{equation}
Such condition implies that $\Re \psi$ is unbounded, hence it must not be a characteristic exponent of compound Poisson process, and consequently, $\Re \psi(\xi) >0$ for $\xi \neq 0$. Furthermore, since $e^{-x}\leq (1+x)^{-1}$ for $x\geq 0$, we obtain that $\Big\lvert e^{-t\psi(\cdot)} \Big\rvert = e^{-t\Re \psi(\cdot)}$ is integrable. Thus, by the Fourier inversion formula, the transition density $p(t,\cdot)$ of $X_t$ exists for all $t>0$ and is given by
\[
p(t,x) = \frac{1}{2\pi} \int_{\R} \Re \Big[ e^{-t\psi(\xi)-i\xi x} \Big]\,d\xi, \quad x \in \R.
\]
If, following Pruitt \cite{Pruitt81}, we define \emph{the concentration function} by setting for $r>0$,
\[
h(r) = \frac{\sigma^2}{r^2} + \int_{\R} \Bigg( 1 \wedge \frac{|s|^2}{r^2} \Bigg)\nu({ds}),
\]
then by \cite[Lemma 4]{TG14}, for all $x \in \R$, 
\begin{equation*}
	h(1/|x|) \approx \psi^*(|x|).
\end{equation*}
Lastly, also due te Pruitt we introduce the compensated drift part by setting 
\[
b_r = \gamma+\int_{\R} z \big( \ind_{|z|<r}-\ind_{|z|<1} \big)\,\nu(dz).
\]
We now turn to introduction of basic objects from the potential theory. For any $x \in \R$, $\P^x$ and $\E^x$ will denote the distribution and the expectation for the process $\mathbb{X}+x$, with the standard notation $\P^0=\P$ and $\E^0=\E$. We also write $\E^x [A;Y]=\E^x \ind_{A}Y$. By $\tauD$ we denote \emph{the first exit time} from an open set $D$, i.e.
\[
\tauD = \inf\{ t>0\colon X_t \notin D \}.
\]
For a closed set $F$ we define \emph{the first hitting time} of $F$ as the first exit time from its complement $F^c$, that is $T_F = \tau_{F^c}$. If $F=\{b\}$ is a singleton, then slightly abusing the notation, we write $T_b = T_{\{b\}}$.
For $\lambda>0$ we let $u^{\lambda}$ be \emph{the $\lambda$-potential kernel}, that is the Laplace transform of $p(\cdot,x)$:
\[
u^{\lambda}(x) = \int_0^{\infty} e^{-\lambda t} p(t,x)\,dt.
\]
If $u^{\lambda}$ exists for $\lambda=0$ then we set $u^0=u$ and then the process is transient (see \cite[Theorem I.17]{Bertoin96}). In general, this needs not be the case (take for example the classical stable process with stability index $\alpha>1$, which is recurrent). Nevertheless, the condition
\[
\int_0^{\infty} \frac{d\xi}{1+\Re \psi(\xi)}<\infty,
\] 
together with \cite[Theoreme 7 and 8]{Bretagnolle71}, implies that $0$ is regular for itself, that is
\[
\P^0 \big( T_0=0 \big)=1.
\]
That in turn, combined with \cite[Theorem 1]{Kesten69}, yield that $\P^x (T_0<\infty)>0$ for any $x \in \R$. 

Let us also set
\[
\kappa = \lim_{\lambda \to 0^+} \frac{1}{u^{\lambda}(0)} = \lim_{\lambda \to 0^+} \frac{1}{2\pi} \Bigg( \int_{\R} \Re \bigg[ \frac{1}{\lambda+\psi(\xi)} \bigg]\,d\xi \Bigg)^{-1}.
\]
Clearly, $\kappa \in [0,\infty)$. Moreover, from \cite[Theorem I.17]{Bertoin96} it follows that the process is transient if $\kappa>0$.

%If we define for $\lambda \geq 0$,
%\[
%s^{\lambda}(x) = \E^x e^{-\lambda T_0}, \quad x \in \R,
%\]
%then 
By \cite[Corollary II.19 and Theorem II.19]{Bertoin96} we get that $x \mapsto \E^x e^{-\lambda T_0}$ is continuous with respect to $x$, and for any $x \in \R$ we have
\begin{equation}\label{eq:1}
u^{\lambda}(x) = u^{\lambda}(0)\E^{-x} e^{-\lambda T_0}.
\end{equation}
Note that for recurrent processes we have $u(x)=\infty$ for all $x \in \R$. Instead one can define \emph{the compensated $\lambda$-potential kernel} by setting for $\lambda >0$
\[
K^{\lambda}(x) = u^{\lambda}(0)-u^{\lambda}(x), \quad x \in \R.
\]
In view of \eqref{eq:1} we get that $K^{\lambda} \geq 0$ for all $\lambda \geq 0$.

The next natural move would be to pass with $\lambda$ to $0$ and define \emph{the compensated potential kernel}
\[
K(x) = \lim_{\lambda \to 0^+} K^{\lambda}(x), \quad x \in \R.
\]  
In general, however, it is not easy to show even the existence of $K$, let alone its further properties. For some elaboration on that subject, including special cases when $K$ can be well-defined, see Section \ref{sec:asymptotics}. \emph{If} it does exist then one can express the probability of not hitting the origin in terms of $K$ and $\kappa$.
\begin{proposition}\label{prop:5}
	Suppose that $K$ exists. Then
	\[
	\P^x (T_0=\infty)=\kappa K(-x).
	\]
	If $\kappa=0$ then $\P^x (T_0<\infty)=1$ for all $x \in \R$.
\end{proposition}
It follows that if $K$ exists and $\P^x (T_0<\infty)=1$ then $\mathbf{X}$ is recurrent.
\begin{proof}
	Observe that
	\[
	\lim_{\lambda \to 0^+} \E^x e^{-\lambda T_0} = \P^x (T_0 < \infty).
	\]
	On the other hand,
	\[
	\E^x e^{-\lambda T_0} = \frac{u^{\lambda}(-x)}{u^{\lambda}(0)} = 1- \frac{u^{\lambda}(0)-u^{\lambda}(-x)}{u^{\lambda}(0)} \to 1-\kappa K(-x),
	\]
	as $\lambda \to 0^+$.
\end{proof}
Instead of $K$, let us consider the symmetrized $\lambda$-potential kernel
\[
H^{\lambda}(x) = K^{\lambda}(x) + K^{\lambda}(-x).
\]
By \cite[Theorem II.19]{Bertoin96},
\[
H^{\lambda}(x) = \frac{1}{\pi}\int_{\R} \big(1-\cos xs\big)\Re \Bigg[ \frac{1}{\lambda+\psi(s)} \Bigg]\,ds.
\]
Under \eqref{eq:main_assumption} one can show that \emph{the symmetrized potential kernel}
\[
H(x) = \lim_{\lambda \to 0^+} H^{\lambda}(x), \quad x \in \R,
\]
is well-defined and
\[
H(x) = \frac{1}{\pi}\int_{\R} \big(1-\cos xs\big)\Re \Bigg[ \frac{1}{\psi(s)} \Bigg]\,ds, \quad x \in \R.
\]
Proceeding as in the proof of \cite[Proposition 2.2]{GR17} one may prove that $H$ is subadditive on $\R$. 

We will often assume the so-called global weak lower scaling condition on the real part of the characteristic exponent with the scaling index strictly bigger than $1$. We will say that a function $f$ satisfies \emph{the global weak lower scaling condition}, if there are $\alpha \in \R$ and $c \in (0,1]$ such that for all $u>0$ and $\lambda \geq 1$,
\[
f(\lambda u) \geq c\lambda^{\alpha} f(u), \quad u>0.
\]
We will write shortly $f \in \WLSC{\alpha}{c}$. The pair $(\alpha,c)$ will be refered to as \emph{scaling characteristics} or simply \emph{scalings}. Note that by \cite[Lemma 11]{BGR14}, $f \in \WLSC{\alpha}{c}$ for some $\alpha \in \R$ and $c \in (0,1]$ if and only if the function
\[
(0,\infty) \ni x \mapsto x^{-\alpha}f(x)
\]
is almost increasing.  Clearly, the condition $\Re \psi \in \WLSC{\alpha}{c}$ for some $\alpha>1$ and $c \in (0,1]$, implies \eqref{eq:main_assumption}. Moreover, by  \cite[Theorem 3.1]{GS2019}, for all $x \in \R$, 
\begin{equation}\label{eq:2}
	\Re \psi(x) \approx \psi^*(|x|) \approx h(1/|x|).
\end{equation}
We note that under assumptions of the global scaling property of the real part of the characteristic exponent, its control over the imaginary part (see \cite[Lemma 12]{TG19}) and vanishing of the first moment, i.e. $\E X_1=0$, one can show that $\int_0^1 \Re(1/\psi(\xi))d\xi=\infty$ and, in view of \cite[Theorem I.17]{Bertoin96}, conclude that $X_t$ is recurrent, and consequently, $u(x)=\infty$ for all $x \in \R$. As the processes we study in this paper often satisfy such assumptions, we are in dire need of some object alternative to the (infinite) potential kernel. The symmetrized compensated potential kernel $H$ can be of usage here, but it appears that the (ordinary) compensated kernel $K$ is more appropriate for description of hitting probability behaviour. The only problem is that in general we are not able to determine whether it exists. We devote Sections \ref{sec:asymptotics} and \ref{sec:estimates} to the detailed discussion on the subject.

We note one simple yet crucial observation.
\begin{proposition}\label{prop:1}
	Suppose that $\E X_1=0$ and $\Re \psi \in \WLSC{\alpha}{\chi}$ for some $\alpha>1$ and $\chi \in (0,1]$. There is $c \geq 1$ such that for all $r>0$,
	\[
	c^{-1}\frac{1}{rh(r)} \leq H(r) \leq c \frac{1}{rh(r)}.
	\]
	In particular, $H \in \WLSC{\alpha-1}{\tilde{\chi}}$ for some $\tilde{\chi} \in (0,1]$.
\end{proposition}
\begin{proof}
	Observe that by \cite[Lemma 12]{TG19}, for any $r>0$,
	\[
	H(r) \approx \int_0^{\infty} (1-\cos rs) \frac{1}{\Re \psi(s)}\,ds.
	\]
	Now, the claim follows by \cite[Lemma 13]{TG19} and \cite[Lemma 2.3]{GS2019}.
\end{proof}

Apart from \emph{free processes}, for an open set $D$ one can consider \emph{the process killed when exiting $D$}, denoted by $X_t^D$. Namely, for any measurable function $f$,
\[
\E^x f \big( X_t^D \big) = \E^x \big[t<\tauD ;\,f(X_t)\big], \quad t>0,\,x \in \R.
\]
By analogy, by $p^D(t;\cdot,\cdot)$ we denote its transition density. It is then known that
\[
p^D(t;x,y) = p(t,y-x)-\E^x \big[ t>\tauD;\,p\big(\tauD-t,y-X_{\tauD}\big) \big], \quad t>0,\,x,y \in \R.
\]
By analogy to the free process, for any $x \in \R$ we define \emph{the $\lambda$-potential measure} of $X_t^D$ as
\[
G^{\lambda}_D(x,A) = \int_0^{\infty} e^{-\lambda t} \P^x \big( X_{\tauD} \in A \big) \,dt, \quad A \in \BB_{\R}.
\]
Since $p^D(t;\cdot,\cdot)$ exists for all $t>0$, the $\lambda$-potential measures are absolutely continuous with the density $G_D^{\lambda}$ given by
\[
G_D^{\lambda}(x,y) = \int_0^{\infty} e^{-\lambda t} p^D(t;x,y)\,dt, \quad x,y \in D.
\]
The $0$-Green function is simply called \emph{the Green function} and denoted by $G_D(x,y)$.

Next, we introduce \emph{the harmonic measure of the set $D$}, which describes the distribution of $X_{\tauD}$ started from $x$ on $\{ \tauD < \infty \}$. Namely, for any borel $A \subset \R$,
\[
P_D(x,A) = \P^x \big( X_{\tauD} \in A \big).
\]
If the density kernel of $P_D$ exists then we call it \emph{the Poisson kernel} for the set $D$ and denote by the same symbol $P_D(x,z)$. The celebrated Ikeda-Watanabe formula \cite{IW62} provides a connection between the Poisson kernel and the Green function:
\[
P_D(x,A) = \int_D \nu(A-y)\,G_D(x,dy), \quad A \subset \overline{D}^c.
\]
Finally, we say that a Borel measurable function $u$ is \emph{harmonic} in an open set $D$ with respect to $\mathbf{X}$, if for any bounded open set $B$ such that $\overline{B} \subset D$,
\[
u(x) = \E^x u \big( X_{\tau_B} \big), \quad x \in B.
\]  
If the equality above holds also for $B=D$ then we say that $u$ is \emph{regular harmonic} in $D$ (with respect to $\mathbf{X}$). Also, $u$ is \emph{(regular) coharmonic in $D$} if it is (regular) harmonic in $D$ with respect to the dual process $\widehat{\mathbf{X}}$. Clearly, for symmetric processes a harmonic function is coharmonic and vice versa; in general, that is not necessarily true. We remark that by the strong Markov property, a regular harmonic function is harmonic. Moreover, the Green function for the set $D$ (if it exists) is harmonic in $x$ on $D \setminus \{y\}$. 

Let us note two simple observations.
\begin{proposition}\label{prop:2}
	For any $x,y \in \R$ we have
	\[
	\Go(x,y) \leq  H(x) \wedge H(y).
	\]
\end{proposition}
\begin{proof}
	For any $\lambda>0$ we have
	\begin{align*}
	\Go^{\lambda}(x,y) &= u^{\lambda}(y-x)-\E^x e^{-\lambda T_0} u^{\lambda}(y-T_0) \\ &=u^{\lambda}(y-x) - u^{\lambda}(y)h^{\lambda}(x) \\ &= -K^{\lambda}(y-x) + K^{\lambda}(-x) + K^{\lambda}(y) - \frac{K^{\lambda}(-x)K^{\lambda}(y)}{u^{\lambda}(0)}.
	\end{align*}
	Recall that $K^{\lambda} \geq 0$. By the proof of \cite[Proposition 2.2]{GR17} we get that $K^{\lambda}$ is subadditive on $\R$. Thus, $K^{\lambda}(y) \leq K^{\lambda}(x) + K^{\lambda}(y-x)$ and consequently,
	\[
	\Go^{\lambda}(x,y) \leq -K^{\lambda}(y-x)+K^{\lambda}(-x)+K^{\lambda}(y) \leq K^{\lambda}(x)+K^{\lambda}(-x).
	\]
\end{proof}
\begin{proposition}\label{prop:3}
	For any $|x| \in (0,R)$ we have
	\[
	\E^x \big[ \tauRR \wedge T_0 \big] \leq 2RH(x).
	\]
\end{proposition}
\begin{proof}
	By Proposition \ref{prop:2},
	\[
	\E^x \big[ \tauRR \wedge T_0 \big] = \int_{-R}^R G_{(-R,0)\cup (0,R)}(x,y)\,dy \leq \int_{-R}^R \Go(x,y)\, dy \leq 2RH(x).
	\]
\end{proof}

\subsection{Fluctuation theory}\label{subsec:fluctuation}
Before embarking on further results we need some introduction from the fluctuation theory of L\'{e}vy processes. The general reference here is the book of Bertoin \cite{Bertoin96}.

First, let us observe that the condition \eqref{eq:main_assumption} implies that $\mathbf{X}$ is of unbounded variation. Indeed, suppose the converse; then $\mathbf{X}$ can be written as a difference of two subordinators. It follows that $\Re \psi$ has at most linear growth and hence,
\[
\int_0^{\infty} \frac{d\xi}{1+\Re \psi(\xi)}=\infty,
\]
which is a contradiction. Thus, by \cite{Rogozin68}, $0$ is regular for half-lines $(-\infty,0)$ and $(0,\infty)$. Now, let $\mathbf{L}=(L_t\colon t \geq 0)$ be a local time at $0$ for the process $\mathbf{S}-\mathbf{X}$, where $S_t = \sup_{s \leq t}X_s$, and $\mathbf{L}^{-1}$ --- its right-continuous inverse, called \emph{the ascending ladder time process}. Next, we define $\mathbf{H}$ by setting $H_t = S_{L^{-1}_t}=X_{L^{-1}_t}$ on $\{ L^{-1}_t<\infty \}$ and $H_t=\infty$ otherwise. $\mathbf{H}$ is a (possibly killed) subordinator called \emph{the ascending ladder height process}. 
%The pair $(\mathbf{L}^{-1},\mathbf{H})$ is called \emph{the ladder process}. It is a bivariate subordinator with the Laplace exponent denoted by $\kappa(z,t)$. Similarly, the Laplace exponent for the dual process $\widehat{\mathbf{X}}=-\mathbf{X}$ is denoted by $\widehat{\kappa}(z,t)$. Some useful results concerning $\kappa$ and $\widehat{\kappa}$ may be found in \cite{TG19}.

Next, we introduce the renewal function $V$ as a potential measure of the interval $[0,x]$ for the ascending ladder height process, i.e. $V(x) = \int_0^{\infty} \P(H_s \leq x)\,ds$, with the convention $V(x)=0$ for $x < 0$. Similarly, $\widehat{V}(x) = \int_0^{\infty} \P (\widehat{H}_s \leq x)\,ds$ and $\widehat{V}(x)=0$ for $x < 0$. Clearly, if $\mathbf{X}$ is symmetric, then $\widehat{\mathbf{X}} = \mathbf{X}$ and consequently, $\widehat{V}=V$. Directly from the definition of $V$ and $\widehat{V}$ we conclude that both are subadditive. Moreover, by \cite[Theorem 1 and 2]{Silverstein80}, $\widehat{V}$ and $\widehat{V}'$ are harmonic on $(0,\infty)$, but $V$ and $V'$ are coharmonic on $(0,\infty)$. In fact, also from \cite{Silverstein80} we have that both $V'$ and $\widehat{V}'$ are positive on $(0,\infty)$, hence both $V$ and $\widehat{V}$ are actually strictly increasing on $(0,\infty)$. One feature that will be substantial in Section \ref{sec:harmonic} is the fact that the Green function for the positive half-line can be represented in terms of $V$ and $\widehat{V}$ in the following way:
\begin{equation}\label{eq:green_f}
G_{(0,\infty)}(x,y) = \int_0^x \widehat{V}'(u)V'(y-x+u)\,du, \quad y>x>0.
\end{equation}
That identity is provided by \cite[Theorem VI.20]{Bertoin96}.

A number of helpful results concerning $\widehat{V}$, which is of greater importance and usage in the case of non-symmetric processes, are derived in \cite{TG19}. Below we provide a sharp estimate on the probability that the process, when exiting the interval $(0,R)$, chooses the right end. Such result seem to be interesting in and of itself, as we provide sharp two-sided bound, which is an analogue of the estimate for the symmetric case (see e.g. \cite[Proposition 3.7]{GR12}).
\begin{proposition}\label{prop:prob_exit_sharp}
	Suppose that $\E X_1 = 0$ and $\Re \psi \in \WLSC{\alpha}{\chi}$ for some $\alpha>1$ and $\chi \in (0,1]$. Then there is $c \in (0,1]$ such that for any $R>0$ and $0<x<R$,
	\[
	c \frac{\widehat{V}(x)}{\widehat{V}(R)} \leq \P^x \big( \tauoR < \tauoinf \big) \leq \frac{\widehat{V}(x)}{\widehat{V}(R)}.
	\]
\end{proposition}
\begin{proof}
	Fix $R>0$ and let $x \in (0,R)$. From the proof of \cite[Theorem 9]{TG19} we get the following results. First, 
	\[
	\P^x \big( \tauoR < \tauoinf \big) \leq \frac{\widehat{V}(x)}{\widehat{V}(R)}.
	\]
	Next, in view of \cite[Proposition 4 and Theorem 6]{TG19}, there is $c_1$ such that
	\[
	c_1 \frac{\widehat{V}(x)}{\widehat{V} \big( h^{-1}(1/t) \big)} - \frac{\widehat{V}(x)V(R)}{t} \leq \P^x \big( \tauoR < \tauoinf \big).
	\]
%	Finally, for $t_0 = 1/(2C_2(h(R)+|b_R|/R))$, where $C_1$ is taken from Pruitt \cite{Pruitt81}, we have that  $\widehat{\kappa}(1/t_0,0)\widehat{V}(R) \geq 1/4$.

%	Now, by \cite[Corollary 3]{TG19}, there is $c_2 \in (0,1]$ such that for any $t>t_0$,
%	\[
%	\widehat{\kappa} \bigg(\frac{1}{t},0\bigg) \geq c_2\bigg(\frac{t_0}{t}\bigg)^{\rho} \widehat{\kappa} \bigg( \frac{1}{t_0},0 \bigg).
%	\]
%	Let $t = at_0$, where $a >1$. 
%	Now, in view of \cite[Corollary 5]{TG19}, there are $c_2,c_3\geq 1$ such that
%	\[
%	|b_R| \leq c_2Rh(R), \quad R>0,
%	\]
%	and consequently,
	Now we specify $t>0$. Set $t=a/h(R)$, where $a \geq 1$. By \cite[Lemma 2.3]{GS2019}, there is $c_2 \geq 1$ such that
	\[
	h^{-1}(1/t)=h^{-1} \big( h(R)/a \big) \leq c_2 a^{1/\alpha}R.
	\]
	Taking into account subadditivity and monotonicity of $\widehat{V}$, we infer that
	\[
	\widehat{V} \big( h^{-1}(1/t) \big) \leq \widehat{V} \big( c_2 a^{1/\alpha}R \big) \leq 2c_2 a^{1/\alpha} \widehat{V}(R),
	\]
	if only $a \geq (2c_2)^{-\alpha}$. Furthermore, by \cite[Corollary 5]{TG19}, there is $c_3 \geq 1$ such that
	\[
	h(R) \leq \frac{c_3}{V(R)\widehat{V}(R)}.
	\]
	It follows that
	\begin{align*}
	\P^x \big( \tauoR < \tauoinf \big) &\geq 2^{-1}c_1 c_2^{-1} a^{-1/\alpha} \frac{\widehat{V}(x)}{\widehat{V}(R)} - \frac{c_3}{a} \frac{\widehat{V}(x)}{\widehat{V}(R)} = \frac{\widehat{V}(x)}{\widehat{V}(R)} \frac1a \Big( 2^{-1}c_1 c_2^{-1}a^{(\alpha-1)/\alpha} - c_3 \Big) \\ &\geq \frac{1}{a} \frac{\widehat{V}(x)}{\widehat{V}(R)},
	\end{align*}
	if $a$ is big enough, and the claim follows.
\end{proof}

\section{Regular variation}\label{sec:regvar}
In this section we aim to prove that the regular variation of the real part of the characteristic exponent implies regular variation of its imaginary part, if we impose some condition on the structure of the L\'{e}vy measure. First, we recall some basic definitions and properties. The general reference here is the book \cite{RegularVariation89}. 

We say that a function $f\colon (0, \infty) \mapsto (0, \infty)$  is regularly varying at $0$ with parameter $\rho$, if for all $\lambda>0$,
\begin{align*}
\lim_{x \to 0^+} \frac{f(\lambda x)}{f(x)} = \lambda^\rho.
\end{align*}

We say that a function $f\colon (0, \infty) \mapsto (0, \infty)$ is regularly varying at the infinity with parameter  $\rho$, if for all $\lambda>0$,
\begin{align*}
\lim_{x \to \infty} \frac{f(\lambda x)}{f(x)} = \lambda^\rho.
\end{align*}
When $\rho = 0$ we say that a function $f$ is slowly varying at $0$ (or at $\infty$). Clearly, if $f$ is regularly varying at $0$ and at $\infty$ then it satisfies global weak lower scaling condition as well (with the same scaling parameter).

%A generalization of the regular variation is $O$-regular variation, part of which was aforementioned in Section \ref{sec:prelims}. We say that $f$ satisfies the {\it global weak lower scaling condition} ($\psi\in \WLSC{\alpha}{c}$) if there are numbers $\alpha > 0$ (called the index
%of the lower scaling) and $c \in (0, 1]$, such that
%\begin{align}
%f(\lambda x) \geq c \lambda^\alpha f( x), \quad \lambda \geq 1,\,x > 0.
%\end{align}
%Additionally, we say that $f$ satisfies {\it global weak upper scaling condition}  ($\psi\in \WUSC{\alpha}{c}$) if there exist $\beta < 2$ and $c \geq 1$ such that 
%\begin{align}
%f(\lambda x) \leq c \lambda^\beta f(x), \hspace{1cm} \lambda \geq 1, \hspace{0.5cm} x > 0.
%\end{align}
Next, we recall Potter's lemma, which is a very useful property of regularly varying functions (\cite[Theorem 1.5.6]{RegularVariation89}).  If a function $f$ is regularly varying at infinity with a parameter $\rho$, then for any  $C>1, \delta > 0$ exists $X = X(C, \delta),$
such that
\begin{align}\label{potter_lem}
\frac{f(y)}{f(x)} \leq C \max \left[\left(\frac{y}{x} \right)^{\rho + \delta}, \left(\frac{y}{x} \right) ^{\rho - \delta} \hspace{0.1cm}  \right ], \quad x,y>X.
\end{align}
Let $k\colon (0, \infty) \mapsto \R$. We introduce the Mellin transform $\mathcal{M}$ of the function $k$ by
\begin{align*}
\mathcal{M}\{k\}(z) = \int_0^\infty t^{-z} k(t) \frac{dt}{t} ,%= \int_0^\infty u^z k(1/u) \frac{du}{u} 
\end{align*}
for $z \in \mathbb{C}$ such that the integral converges.

Next, for $f,g\colon (0, \infty) \mapsto \R$ we define its \emph{Mellin convolution} by
\begin{align*}
g\stackrel{m}{*} f(x)= \int_0^\infty g(x/t)f(t)\frac{dt}{t}, \quad x>0. 
\end{align*}
Finally, by
\begin{align*}
f\stackrel{s}{*} g(x) = \int_0^\infty f(x/t)dg(t), \quad x>0. 
\end{align*}
we denote the \emph{Mellin-Stieltjes convolution} for functions $f$ and $g$ such that the Stieltjes integral is well defined.
\begin{proposition}\label{funkcja-miara}
Let $\alpha\in(0,2)$. $\Re \psi $ is regularly varying at infinity (at $0$) with the exponent $\alpha$ if and only if $t\mapsto {\nu}(\{s \colon|s|\geq t\})$ is regularly varying at 0 (at infinity) with the exponent $-\alpha$. Moreover,
\begin{align*}
\nu(\{s\colon|s|\geq t\})  \cong  \frac{\Gamma(1 + \alpha) }{\mathrm{B}\Big(1 - \frac{\alpha}{2},1 + \frac{\alpha}{2}\Big)} \Re\psi(1/t), \quad t \to 0^+ \quad (t\to \infty).
\end{align*} 
\end{proposition}

\begin{proof}
We compute the Laplace transform of the function $\Re \psi$. By Fubini's theorem,
\begin{align*}
\mathcal{L}\{\Re {\psi}\}(\lambda) = & \int_0^\infty e^{-\lambda t} \int_ \R (1-\cos tx)\,{\nu}(dx)\,dt = \int_\R \int_0^\infty e^{-\lambda t}(1- \cos tx)\,dt\,{\nu}(dx) \\
				      = & \frac{1}{\lambda}  \int_\R \frac{x^2}{\lambda^2 + x^2}\,{\nu}(dx).
\end{align*}
Next,
\begin{align*}
\frac{\lambda}{2} \mathcal{L}\{\Re \psi\} (\lambda) = &\frac{1}{2} \bigg( \int_{-\infty}^{0}\frac{x^2}{\lambda^2 + x^2}\,\nu(dx) +  \int_{0}^{\infty} \frac{x^2}{\lambda^2 + x^2}\,{\nu}(dx) \bigg) \\ = & -\int_{-\infty}^0   \int_x^0 \frac{\lambda^2 t}{ (\lambda^2 +t^2)^2 }\,dt \,\nu(dx) + \int_0^\infty   \int_0^x \frac{\lambda^2 t}{ (\lambda^2 +t^2)^2 }\,dt\,\nu(dx) \\ = & -\int_{-\infty}^0 \frac{ (\lambda t)^2}{ (\lambda^2 +t^2)^2 } \nu((-\infty,t])  \frac{dt}{t} + \int_0^\infty \frac{(\lambda t)^2}{ (\lambda^2 +t^2)^2 }  \nu([t,\infty))  \frac{dt}{t} \\ = & \int_0^\infty \frac{  (\frac{\lambda}{t})^2 }{{ (1 + (\frac{\lambda}{t})^2)}^2 } \big( {\nu}((-\infty,-t]) +{\nu}([t,\infty)) \big)\frac{dt}{t}.
\end{align*}

Assume that $\Re \psi$ varies regularly at infinity with exponent $\alpha\in (0,2)$. Let us define $\nu_1=\nu|_{|x|\leq 1}$ and the characteristic exponent $\psi_1$ corresponding to the triplet $(0,0,\nu_1)$.  We have $\Re\psi\cong\Re\psi_1$ at infinity. Indeed, since 
\[
0\leq \int_{|x|>1}(1-\cos(xz))\,\nu(dx)\leq 2 \nu\big(\{x\colon |x|>1\}\big),\quad z\in \mathbb{R},
\]
and $\lim_{z\to\infty} \Re\psi(z)= \infty$, we get 
\[
\frac{\Re\psi(z)-\Re\psi_1(z)}{\Re\psi(z)} \to 0,
\]
as $z \to \infty$. Next, using the Abel theorem \cite[Theorem 1.7.1]{RegularVariation89}
 one can observe that
\begin{align*}
\lambda^{-1}\mathcal{L}\{\Re{\psi}_1\}(1/\lambda) \cong \Re\psi(\lambda) \Gamma(1 + \alpha), \quad \lambda \to \infty.
\end{align*}

Let $g(t)= \nu_1(\{s\colon |s|\geq 1/t\}),\, t > 0$  and  $k(t) = \frac{t^2}{{(1+t^2 )}^2}$. Observe that  the Laplace transform of $\Re\psi_1$ is the Mellin convolution of $k$ and $g$:
\begin{align*}
\frac{1}{2\lambda} \mathcal{L}\{\Re\psi_1\} (1/\lambda) = \int_0^\infty k(1/(t\lambda)){g}(1/t)\frac{dt}{t} =\int_0^\infty k(t/\lambda){g}(t)\frac{dt}{t}= k \stackrel{m}{*} {g}(\lambda),
\end{align*}
where in the last equality we used $k(t)=k(1/t)$ for  $t>0$.

To prove that $g(t)$ is regularly varying function, we will use \cite[Theorem 4.9.1]{RegularVariation89} for the function $g_1(t) = \int_t^\infty {g}(s)\frac{ds}{s}$ and convolution $ k \stackrel{s}{*} g_1(\lambda)= k \stackrel{m}{*} g(\lambda)$. Set $\sigma$ such that $ -2 <\sigma <-\alpha$ and $\tau = 0$. Observe that
\begin{align*}
\|k\|_{\sigma, \tau} = & \sum_{-\infty < n <\infty} \max(e^{-\sigma n}, e^{-\tau n}) \sup_{e^n \leq x \leq e^{n+1}} \big| k(x)\big| \\
\leq  &\ \ \sum_{n \leq -1}  e^{2n} +  \sum_{n \geq 0}  \frac{e^{-\sigma n}}{e^{2n}} < \infty.
\end{align*}
Moreover, by \cite[Table 1.2 (2.19)]{MR0352890},
\[
\mathcal{M}\{k\}(z)=\int^\infty_0\frac{t^2}{(1+t^2)^2}t^{z-1}\,dt=\frac{1}{2}\int^\infty_0\frac{1}{(1+s)^2}s^{(z/2+1)-1}\,ds= \frac{\Gamma\big( 1-\frac{z}{2} \big)\Gamma\big( 1 + \frac{z}{2} \big)}{2}.
\] 
The function $\Gamma$ does not have any roots, so the Wiener condition $\mathcal{M}\{k\}(z) \neq 0$ is satisfied. Notice that $g_1$ is non increasing on $(0, \infty)$ and the function $g_1(t)$ is zero on $(0,1)$. Hence, $g_1(t) = \mathcal{O}(t^\sigma)$ at $0^+$. That is, the kernel $k$ and the function $g_1$ satisfy assumptions of \cite[Theorem 4.9.1]{RegularVariation89}, hence,
\begin{align*}
g_1(t) \cong \frac{\Gamma(1+\alpha)}{\mathrm{B}\Big(1 - \frac{\alpha}{2},1 + \frac{\alpha}{2}\Big)} \frac{\Re\psi(t)}{\alpha}, \quad t \to 0^+.
\end{align*}
By monotone density theorem \cite[Theorem 1.7.2]{RegularVariation89} 
and the fact that $g(1/t)\cong \nu(\{s\colon |s|\geq t\})$ as $t$ goes to $0^+$, we obtain that 
\begin{align*}
\nu(\{s\colon |s|\geq t\})  \cong  \frac{\Gamma(1 + \alpha) }{\mathrm{B}\Big(1 - \frac{\alpha}{2},1 + \frac{\alpha}{2}\Big)} \Re\psi(1/t), \quad t \to 0^+.
\end{align*}
In particular, $t\mapsto \nu(\{s\colon |s|\geq t\})$ is regularly varying function at $0$ with index $-\alpha$. 

Now assume that $t\mapsto \nu(\{s \colon |s|\geq t\})$ is regularly varying function at $0$ with index $-\alpha$. Again, instead of $\psi$ one can consider $\psi_1$. Since 
\[
\Re\psi_1(z)= z\int^1_0 \sin(xz) \nu(\{s\colon |s|\geq x\})\,dx=\int^z_0 \sin(xz) \nu(\{s\colon |s|\geq x/z\})\,dx,
\]
one can use Potter's Lemma to justify that
\[
\lim_{z\to\infty}\frac{\Re\psi_1(z)}{\nu(\{s\colon|s|\geq 1/z\})}=\int^\infty_0\frac{\sin x}{x^\alpha}\,dx,
\]
which finishes the proof in this case. 

If $\Re\psi$ varies regularly at 0 one can modify the above prove to obtain the behaviour of  the tail of $\nu$ at infinity.
\end{proof}
We remark that in fact, equivalence of regular variation of $\Re\psi$  at 0 and regular variation of the tail of $\nu$ at infinity can be easily obtained from \cite{MR0231423}. Our proof, however, works in both cases.

If  the Lévy measure is of the special form
\begin{align}
\nu(dx) = C_d \mathds{1}_{\{x<0\}}\nu_0(dx) + C_u \mathds{1}_{\{x>0\}}\nu_0(dx), \label{gest-miar}
\end{align} 
where $\nu_0(dx)$ is a symmetric L\'{e}vy measure, the theorem above provides the behaviour of the one-sided tail of $\nu$ as well. For instance, if $\Re\psi$ is
regularly varying at $0$ with the exponent  $\alpha \in (0,2)$, then
\begin{align*}
\nu_0([t,\infty))  \cong \frac{1}{C_u + C_d}  \frac{\Gamma(1 + \alpha) }{\mathrm{B}\Big(1 - \frac{\alpha}{2},1 + \frac{\alpha}{2}\Big)} \Re\psi(1/t), \quad t \to \infty.
\end{align*}
This is the case for stable processes, where even the equality holds true.
\begin{lemma}\label{finite_moment}
If $\Re\psi$ varies regularly
at $0$ with an exponent $\alpha>1$, then
\begin{align*}
\int_{(-1,1)^c} |x| \,\nu(dx) < \infty.
\end{align*}
\end{lemma}

\begin{proof}
If $\Re\psi$ varies regularly at $0$ with a positive exponent $\Re\psi\cong \psi^*$ near $0$. Hence $h(r)\approx \Re\psi(1/r)$ for large $r>R_0$. By the Potter Lemma we get
\begin{align*}
\int_{|x|>1}|x|\,\nu(dx)&=\int^\infty_0\nu\big(\{s\colon|s|>1\vee u\}\big)\,du \\&\leq \int^\infty_0 h(1\vee u)\,du\leq h(1)R_0+c\int^\infty_{R_0}\Re\psi(1/u)\,du<\infty.
\end{align*}
The claim follows immediately.
\end{proof}

\begin{lemma}\label{lem:measure_function}
Assume that $f(s)$, $s \geq 0$, is a function which is non-negative, regularly varying at $\infty$ with parameter $-\alpha,$ where $\alpha \in (1,2)$, and such that $\int_0^\infty (1 \wedge s^2)f(s)\,ds < \infty$.  Then the transformation
\begin{align*}
x \mapsto \int_0^\infty (1-\cos xs)f(s)\,ds, \quad x \geq 0,
\end{align*}
is regularly varying at $0$ with the parameter $\alpha-1$ and satisfies
\begin{align*}
\int_0^\infty (1-\cos xs)f(s)\,ds \cong -\frac{f(1/x)}{x}\frac{\pi}{2\Gamma(\alpha)\cos\left(\frac{\pi\alpha}{2}\right)}, \quad x \to 0^+.
\end{align*}
\end{lemma}

\begin{proof}	
Let $x>0$. By the Potter Lemma \eqref{potter_lem}, one can set $A$ such that for $s > A$, there exist  $-\alpha < -\hat{\alpha} < -1$ and $M$ such that $\frac{f(s/x)}{f(1/x)} < M s^{-\hat\alpha}$. Notice that
\begin{align*}
\frac{x}{f(1/x)}\int_0^\infty (1-\cos xs)f(s)\,ds & = \frac{x}{f(1/x)} \int_0^A (1-\cos xs)f(s)\,ds + \frac{x}{f(1/x)} \int_A^\infty (1-\cos xs)f(s)\,ds \\
				    & = \frac{x}{f(1/x)} \int_0^A (1-\cos xs)f(s)\,ds + \int_{Ax}^\infty (1-\cos s) \frac{f(s/x)}{f(1/x)}\,ds.
\end{align*}
Again by the Potter Lemma, for constant $C$ and $\rho < 2$ there exists $\epsilon > 0$
such that $Cx^\rho \leq f(1/x)$, when $0< x <\epsilon$.
Then we can observe that
\begin{align*}
\frac{x}{f(1/x)} \int_0^A (1-\cos xs)f(s)\,ds \leq C\frac{x^3}{x^\rho } \int_0^A s^2f(s)\,ds \to 0,
\end{align*} 
as $x \to 0^+$. Moreover, for $|s| > Ax$,
\begin{align*}
(1-\cos s) \frac{f(s/x)}{f(1/x)} \leq M(1-\cos s) s^{-\hat\alpha}.
\end{align*}
It allows us to make use of the dominated convergence theorem. Therefore,
\begin{align*}
\lim_{x \to \infty} \frac{x}{f(1/x)}\int_0^\infty (1 - \cos xs)f(s)\,ds = & \int_0^\infty (1-\cos s) s^{-\alpha} \,ds = \frac{\pi}{2\Gamma(\alpha)\sin\Big(\frac{\pi(\alpha-1)}{2}\Big)} \\ = &-\frac{\pi}{2\Gamma(\alpha)\cos\big(\frac{\pi\alpha}{2}\big)},
\end{align*}
where the last inequality follows from \cite[Theorem 14.15]{Sato99}.
\end{proof}

Next theorem is the main result of this section. It shows that with assumption (\ref{gest-miar}), regular variation of the real part of the L\'{e}vy exponent implies regular variation of the imaginary part. This is the case for instance for spectrally one-sided L\'{e}vy processes. 

\begin{theorem}\label{thm:Re_asymp_Im}
Asume that the  L\'{e}vy measure $\nu(dx)$  satisfies (\ref{gest-miar}). Let the real part of the characteristic exponent $\Re \psi(\xi)$ be regularly varying at $0$ with a parameter $\alpha \in (1,2)$. If $\gamma_1 = \gamma - \int_{(-1,1)^c} x\nu(dx) = 0 $ then the imaginary part  $\Im \psi (\xi)$ satisfies
\begin{align*}
\Im \psi (\xi) \cong -\frac{C_u - C_d}{C_u + C_d} \tan\left(\frac{\pi\alpha}{2}\right) \Re\psi(\xi), \quad \xi \to 0^+.
\end{align*}
For $\gamma_1 \neq 0$ we have
\begin{align*}
\Im \psi (\xi) \cong \gamma_1\xi, \quad \xi \to 0.
\end{align*}
\end{theorem}

\begin{proof}
By Proposition \ref{funkcja-miara}, $t \mapsto \nu_0([t, \infty))$ is a regularly varying function  with the exponent $-\alpha$ at $\infty$. More precisely,
\begin{align}\label{eq:3}
\nu_0([t,\infty))  \cong \frac{1}{C_u + C_d}  \frac{\Gamma(1 + \alpha) }{\mathrm{B}\left(1 - \frac{\alpha}{2},1 + \frac{\alpha}{2}\right)} \Re\psi(1/t), \quad t \to \infty.
\end{align}
By Lemma \ref{finite_moment} we know that $\int_{(-1,1)^c} |x|\, \nu(dx) < \infty.$ It allows us to use the representation  (\ref{eq:Im_reprezentation})  of the imaginary part of the function $\psi$.
Let $\xi > 0.$ Observe that
\begin{align*}
\Im \psi(\xi) = &  \gamma_1 \xi+ \int_\R (\xi x - \sin\xi x) \,\nu(dx) \\
		    = & \gamma_1 \xi  + C_u \int_0^\infty \int_0^x (\xi t - \sin\xi t)'\,dt \,\nu_0(dx) - C_d\int_0^\infty \int_x^0 (\xi t - \sin\xi t)'\,dt \,\nu_0(dx) =  \\ 
		    = & \gamma_1 \xi +  C_u\xi\int_0^\infty ( 1 - \cos\xi t)\int_t^\infty  \nu_0(dx) \,dt - C_d\xi\int_0^\infty (1 - \cos\xi t) \int_{-\infty}^t \nu_0(dx) \,dt \\
		    = & \gamma_1 \xi  +  (C_u - C_d) \xi \int_0^\infty  (1 - \cos\xi t) \,\nu_0([t,\infty)) \,dt.
\end{align*}
If $C_u \neq C_d$, by Lemma  \ref{lem:measure_function}, a function $\xi\mapsto (C_u - C_d) \xi \int_0^\infty  (1 - \cos\xi t) \nu_0([t,\infty)) \,dt$ is also regularly varying function at $0$ with the exponent $\alpha$. Assume that $\gamma_1 \neq 0$. Then 
\begin{align*}
\frac{ \gamma_1\xi +  \xi(C_u - C_d) \int_0^\infty  (1 - \cos\xi t) \nu_0([t,\infty)) \,dt }{\gamma_1\xi} \to 1,
\end{align*}
as $\xi \to 0^+$, which follows from the Potter Lemma for $(C_u - C_d) \xi \int_0^\infty  (1 - \cos\xi t) \nu_0([t,\infty)) \,dt$. Then $\Im \psi(\xi)$ is comparable at zero with a linear function. Now, let's assume that  $\gamma_1 = 0$. Observe that $f(s)=\nu([s,\infty))$ satisfies the assumptions of  Lemma \ref{lem:measure_function}. Therefore, if  $\xi \to 0^+$,
\begin{align*}
\xi\int_0^\infty  (1 - \cos\xi t) \nu_0([t,\infty)) \,dt \cong  \nu_0([1/\xi,\infty)) \frac{\pi}{2\Gamma(\alpha)\cos\Big(\frac{\pi\alpha}{2}\Big)}.
\end{align*}
 Using $\Gamma(1-z)\Gamma(z) = \frac{\pi}{\sin\pi z}$ and invoking \eqref{eq:3}, we obtain
\begin{align*}
\Im \psi(\xi) \cong & -\frac{C_u - C_d}{C_u + C_d}  \frac{\alpha\Gamma(\alpha)}{\Gamma\big(1- \frac{\alpha}{2}\big)\Gamma\big(1+ \frac{\alpha}{2}\big)}  \frac{\pi}{2\Gamma(\alpha)\cos\big(\frac{\pi\alpha}{2}\big)} \Re\psi(\xi) \\
		    \cong & -\frac{C_u - C_d}{C_u + C_d}  \frac{\frac{\alpha}{2}\Gamma\big(\frac{\alpha}{2}\big)}{\Gamma\big(\frac{\alpha}{2}\big)\Gamma\big(1- \frac{\alpha}{2}\big)\Gamma\big(1+ \frac{\alpha}{2}\big)}  \frac{\pi}{2\Gamma(\alpha)\cos\big(\frac{\pi\alpha}{2}\big)} \Re\psi(\xi) \\
		    \cong & - \frac{C_u - C_d}{C_u + C_d}  \tan\bigg(\frac{\pi\alpha}{2}\bigg) \Re\psi(\xi), \quad \xi \to 0^+.
\end{align*}
\end{proof}

\section{Asymptotics}\label{sec:asymptotics}
If the process $X_t$ is symmetric then by \cite[Theorem 4.2]{Yano13}, the function $K$ is well-defined. Furthermore, for non-symmetric case by \cite[Proposition 6.1]{Yano13}, existence of first derivatives of $(\Re \psi(\xi))'$ and $( \Im \psi(\xi))'$ such that 
\begin{align}\label{eq:L3}
\int_0^\infty \frac{\left( |(\Re \psi(\xi))'| + |( \Im \psi(\xi))'| \right)(\xi^2 \wedge 1)}{|\psi(\xi)|^2}\,d\xi < \infty.
\end{align}
is sufficient. We remark here that in view of \cite[Theorem 21.9]{Sato99} and discussion at the beginning of Subsection \ref{subsec:fluctuation}, the condition {\bf (L1')} in \cite{Yano13} always implies {\bf (L2)}. 

%Indeed, suppose that $\sigma=0$ and
%\[
%\int_{(-1,1)}|z|\,\nu(dz)<\infty.
%\]
%Then by \cite[Theorem 21.9]{Sato99}, $\mathbf{X}$ is of bounded variation and can be written as a difference of two subordinators. It follows that $\Re \psi$ has at most linear growth and hence,
%\[
%\int_0^{\infty} \frac{d\xi}{1+\Re \psi(\xi)}=\infty.
%\]

Unfortunately, \eqref{eq:L3} does not suit our case and thus, we show the existence of $K$ in several cases. 

\begin{lemma}\label{lem:K_exists}
Assume that $1/(1+\Re\psi)$ is integrable and $\Im\psi\geq 0$ on $(0,\varepsilon)$ for some $\varepsilon>0$. Then $K$ exists and 
\begin{align}\label{eq:K_formula}
K(x) =  \frac{1}{\pi} \int_0^\infty \Re \left [ \frac{1}{\psi(s)} \left(1 - e^{-ixs} \right) \right ]ds.
\end{align}
\end{lemma}
\begin{proof}
Since $e^{-x}\leq (1+x)^{-1}$, $x\geq 0$ we obtain $e^{-\psi}\in L^1(\R)$. By the Riemann-Lebesgue Lemma we have $\Re\psi(\xi)\to\infty$ as $\xi\to\infty$. That implies $1/\Re\psi\in L^1([\delta,\infty))$ for any $\delta>0$ because $\Re\psi(\xi)>0$ for $\xi\neq0$. Next, let us observe that $\Re\psi(\xi)\geq c\xi^2$, $|\xi|\leq 1$, for some $c>0$. Hence,
\[
\int_\R\frac{1-\cos(x\xi)}{\Re\psi(\xi)}d\xi<\infty.
\]
By the dominated convergence theorem, for $x\in\R$,
\[
\lim_{\lambda\to0^+}\int_\R \frac{(1-\cos(x\xi))(\lambda+\Re\psi(\xi))}{|\lambda+\psi(\xi)|^2}\,d\xi= \int_\R \frac{(1-\cos(x\xi))\Re\psi(\xi)}{|\psi(\xi)|^2}\,d\xi,
\]
and
\[\lim_{\lambda\to0^+}\int_{|\xi|\geq \varepsilon \wedge (\pi/|x|)}\frac{\sin(x\xi)\Im\psi(\xi)}{|\lambda+\psi(\xi)|^2}\,d\xi=\int_{|\xi|\geq \varepsilon\wedge (\pi/|x|)}\frac{\sin(x\xi)\Im\psi(\xi)}{|\psi(\xi)|^2}\,d\xi.
\]
For $|\xi|<\varepsilon\wedge (\pi/|x|)$ a function $\xi\mapsto \sin(x\xi)\Im\psi(\xi)$ is non-negative, therefore by the monotone convergence theorem,
\begin{equation}\label{eq:K_near0}
\lim_{\lambda\to0^+}\int_{|\xi|< \varepsilon \wedge (\pi/|x|)}\frac{\sin(x\xi)\Im\psi(\xi)}{|\lambda+\psi(\xi)|^2}\,d\xi=\int_{|\xi|< \varepsilon\wedge (\pi/|x|)}\frac{\sin(x\xi)\Im\psi(\xi)}{|\psi(\xi)|^2}\,d\xi.
\end{equation}
Since $0\leq K^\lambda(x)\leq H(x)<\infty$ for every $\lambda>0$ and $x\in \R$ the above integral is finite. Finally let us notice that the integrand is an even function which ends the proof.
\end{proof}
\begin{corollary}\label{cor:K_moment}
If $1/(1+\Re\psi)$ is integrable, $\E X_1$ exists and $\E X_1\neq 0 $ then  $K$ is well-defined and 
\eqref{eq:K_formula} holds.
\end{corollary}
\begin{proof}
Since $\E |X_1|<\infty$, we have 
\[
\psi(\xi)=\sigma^2\xi^2+ i\gamma_1\xi+\int_\R(1+i\xi z -e^{i\xi z})\,\nu(dz).
\]
A consequence of the dominated convergence theorem is 
\[
\lim_{\xi\to 0^+}\frac{\Im\psi(\xi)}{\xi}=\gamma_1.
\]
Hence, if $\gamma_1=\E X_1\neq 0$ then $\Im\psi$ has a constant sign on $(0,\varepsilon)$ for some $\varepsilon>0$, which finishes the proof due to Lemma \ref{lem:K_exists}.
\end{proof}
\begin{proposition}\label{prop:K_exists_tail}
Assume that $1/(1+\Re\psi)$ is integrable and there is $c>0$ such that
\[
\int_{|z|>r}|z|\,\nu(dz)\leq crh(r),\quad r>1.
\] 
Then $K$ exists and \eqref{eq:K_formula} holds.
\end{proposition}
\begin{remark}
	If $\Re\psi\in \WLSC{\alpha}{\chi}$ for some $\alpha>1$ and $\chi \in (0,1]$, then the assumptions of Proposition \ref{prop:K_exists_tail} are satisfied.
\end{remark}
\begin{proof}
\begin{align} \label{eq:K_IM_h}
\int_{|z|\geq r} |z|\,\nu(dz) \leq c rh(r),
\end{align} 
we have $\E|X_1|<\infty$.
If $\E X_1\neq 0$ we apply Corollary \ref{cor:K_moment} to get the claim of the proposition. Therefore, assume that $\E X_1 = 0$ and then
\[
\psi(\xi)=\sigma^2\xi^2+ \int_\R(1+i\xi z -e^{i\xi z})\,\nu(dz).
\]
By the proof of Lemma \ref{lem:K_exists} it is enough to prove that \eqref{eq:K_near0} holds. Let us  consider a L\'evy measure $\tilde{\nu}(dx)=1_{(0,1)}(x)x^{-5/2}dx+1_{[1,\infty)}\,\nu(dx)$ and a characteristic exponent
\[
\tilde\psi(\xi)=\int_0^\infty(1+i\xi z -e^{i\xi z})\,\tilde{\nu}(dz),\quad \xi\in\R.
\]
Since $\Re\tilde\psi(\xi)\approx |\xi|^{3/2}$, $|\xi|\geq 1$, and 
\[
\Im\tilde\psi(\xi)=\int^\infty(\xi z-\sin(\xi z))\,\tilde\nu(dz)\geq 0,
\]
we can apply Lemma \ref{lem:K_exists} and its proof to obtain finiteness of 
\[
\int^\infty_0\frac{|\sin(x\xi)\Im\tilde\psi(\xi)|}{|\tilde\psi(\xi)|^2}\,d\xi, \quad x\in\R.
\]
Let 
\[
\psi_1(\xi)=\int_0^\infty(1+i\xi z -e^{i\xi z})\,\nu(dz),
\]
and $\psi_2(\xi)=\psi(\xi)-\psi_1(\xi)$. Notice that $\Re\psi_1\approx\Re\tilde\psi$ and $\Im\psi_1\approx\Im\tilde\psi$ on $(0,1)$ and $\Im\psi_1(\xi),\Im\psi_2(-\xi)\,\geq 0$, $\xi\geq 0$. Hence,
\[
\int^1_0\frac{|\sin(x\xi)\Im\psi_1(\xi)|}{|\psi_1(\xi)|^2}\,d\xi<\infty.
\]
But by \eqref{eq:K_IM_h}, we have that $|\Im\psi_1(\xi)|\leq c \Re\psi(\xi)$ for $|\xi|<1$. These implies
\[
\int^1_0\frac{|\sin(x\xi)\Im\psi_1(\xi)|}{|\psi(\xi)|^2}\,d\xi<\infty.
\]
Hence, by the monotone convergence theorem,
\begin{equation*}
\lim_{\lambda\to0^+}\int_{|\xi|< \varepsilon \wedge (\pi/|x|)}\frac{\sin(x\xi)\Im\psi_1(\xi)}{|\lambda+\psi(\xi)|^2}\,d\xi=\int_{|\xi|< \varepsilon\wedge (\pi/|x|)}\frac{\sin(x\xi)\Im\psi_1(\xi)}{|\psi(\xi)|^2}\,d\xi,
\end{equation*} and the limit is finite.
Again by the monotone convergence theorem,
\begin{equation*}
\lim_{\lambda\to0^+}\int_{|\xi|< \varepsilon \wedge (\pi/|x|)}\frac{\sin(x\xi)\Im\psi_2(\xi)}{|\lambda+\psi(\xi)|^2}\,d\xi=\int_{|\xi|< \varepsilon\wedge (\pi/|x|)}\frac{\sin(x\xi)\Im\psi_2(\xi)}{|\psi(\xi)|^2}\,d\xi.
\end{equation*}
Combining the above limits together we obtain \eqref{eq:K_near0}, which ends the proof.
\end{proof}

\begin{corollary}\label{istnienie-K}
Let $\Re \psi$ vary regalarly at $0$ with an exponent $\alpha\in(1,2]$, then $K(x)$ is well-defined and \eqref{eq:K_formula} holds.
\end{corollary}
\begin{lemma} \label{asym}
Assume that $\Re\psi$ varies regularly at $0$ with an exponent $\alpha \in (1,2]$ and 
\begin{align*} 
\frac{\Im \psi(\xi) }{\Re\psi(\xi)} \to C_I, \quad \xi \to 0^+, 
\end{align*}
for some  $C_I \in\R$.
Then
\begin{align*}
\lambda u^{\lambda}(0) \cong   (\Re\psi)^{-1}(\lambda)C(\alpha,C_I),%}{\pi} \int_0^\infty \frac {1 +w^\alpha}{(1 +  w^\alpha)^2 + (C_I w^\alpha)^2}dw, 
\quad \lambda \to 0^+,
\end{align*}
where
\[
C(\alpha,C_I)=\frac{\cos(\mathrm{arctg}(C_I)/\alpha)}{\alpha(1+C_I^2)^{1/(2/\alpha)}\sin(\pi/\alpha)}.
\]
\end{lemma}

\begin{proof}
Denote $\theta(\xi) := \Re \psi (\xi)$ i $\omega(\xi) := \Im \psi (\xi)$. We have 
\begin{align*}
u^{\lambda}(0) = \frac{1}{\pi} \int_ 0^\infty \frac {\lambda + \theta(\xi)} {(\lambda + \theta(\xi))^2 + \omega(\xi)^2}\,d\xi.
\end{align*}
Notice that, for any $\delta>0$,
\[
\textrm{I}_1(\lambda):=\int_\delta^\infty \frac {\lambda + \theta(\xi)} {(\lambda + \theta(\xi))^2 + \omega(\xi)^2}\,d\xi\leq \int_ {\delta}^{\infty} \frac{1} {\theta(\xi)}\,d\xi < \infty.
\]
Since $\alpha>1$, we obtain  $\frac {\lambda}{\theta^{-1}(\lambda)}\mathrm{I}_1(\lambda) \to 0$ as $\lambda \to 0^+$, hence it does not have impact on the asymptotic behaviour.

Set $\mathrm{I}_2(\lambda):=\pi u^\lambda(0)-\mathrm{I}_1(\lambda)$.
Since $\theta$ is continuous function we have $\theta(\theta^{-1}(s))=s$. Hence,
\begin{align}\nonumber
 \frac{\lambda}{\theta^{-1}(\lambda)} \mathrm{I}_2(\lambda)  =& \int_ 0^{\frac{\delta}{\theta^{-1}(\lambda)}} \frac {1 + \theta\left(\theta^{-1}(\lambda)w \right)/\lambda} {\left(1 + \theta\left(\theta^{-1}(\lambda)w\right)/\lambda \right)^2 + \left(\omega\left(\theta^{-1}(\lambda)w\right)/\lambda \right)^2}\,dw \\
 =& \int_0^{\frac{\delta}{\theta^{-1}(\lambda)}} \frac {1 + \frac {\theta(\theta^{-1}(\lambda)w)}{\theta(\theta^{-1}(\lambda))}}{\left(1 + \frac {\theta \left(\theta^{-1}(\lambda)w \right)} {\theta\left(\theta^{-1}(\lambda) \right) }\right)^2 + \left(\frac {\omega(\theta^{-1}(\lambda)w)} { \theta(\theta^{-1}(\lambda)w)} \frac {\theta(\theta^{-1}(\lambda)w)} {\theta(\theta^{-1}(\lambda))} \right)^2}\,dw  .\label{lem:asymp_u_0_1}
\end{align}

Now we will choose $\delta$. Set $\rho$ such that $1 < \rho < \alpha$. By the Potter Lemma  there exists $\delta>0$ such that for $\lambda<\theta(\delta)$, $s>1$ and  $\psi^{-1}(\lambda)s < \delta$, 
\begin{align} \label{Potter}
\frac{\theta \left(\theta^{-1}(\lambda)s \right)}{\theta \left(\theta^{-1}(\lambda) \right)} \geq \frac{1}{2}s^{\rho}.
\end{align}
Thus, integrand in \eqref{lem:asymp_u_0_1}  is dominated by
\begin{align*}
 \frac {1}{1 + \frac {\theta(\theta^{-1}(\lambda)w)}{\theta(\theta^{-1}(\lambda))}} \leq \frac{2}{1+w^\rho/2},\quad w\leq \delta/\theta(\lambda).
\end{align*}
By the dominated convergence theorem,
\begin{align*}
\lim_{\lambda \to 0^+} \frac {\lambda}{\theta^{-1}(\lambda)}u^{\lambda}(0) = \frac{1}{\pi}\int_0^\infty \frac {1 +  w^\alpha}{(1 +  w^\alpha)^2 + (C_I w^\alpha)^2}\,dw,
\end{align*}
which ends the proof since the limit is equal to $u^1(0)$ for stable processes (see \cite[page 389]{Port67}).
\end{proof}

\begin{theorem} \label{gtw}
Assume that $\Re\psi(\xi)$ varies regularly at $0$ with an exponent $\alpha\in(1,2]$ and $\lim_{\xi\to0^+}\Im \psi(\xi)/\Re\psi(\xi)= C_I$. Let $B$ be a compact set such that $0 \in B$. Then, for $x \in \R$,
\begin{align*}
\lim_{t \to \infty} t(\Re\psi)^{-1} (1/t) \P^x(T_B > t) = \frac{1}{C(\alpha,C_I)\Gamma(1/\alpha)}\big(K(-x) - \E^x K(-X_{T_B})\big).
\end{align*}
%where $C =  \frac{1}{2\pi}\int_0^\infty \frac {1 + C_R w^\alpha}{(1 + C_R w^\alpha)^2 + (C_I w^\alpha)^2}dw$.
\end{theorem}

\begin{proof}
We have
\begin{align*}
\mathcal{L} (\P^x(T_B > \cdot))(\lambda) = \frac{1}{\lambda}[ 1 - \E^xe^{-\lambda T_B}].
\end{align*}
By the proof of Proposition \ref{prop:1} and of the fact that $0 \leq G^\lambda_{B^c}(x, 0) \leq G_{{\{0\}}^c}(x, 0)$, we can observe that $G^\lambda_{B^c}(x, 0) = 0$. Hence,
\begin{align}\nonumber
\lambda u^\lambda (0) \mathcal{L} (\P^x(T_B > \cdot))(\lambda) = & u^\lambda(0) - u^\lambda(-x) + u^\lambda(-x) \\
												- & \E^x e^{- \lambda T_B}\big( u^\lambda (0) - u^\lambda(-X_{T_B})\big) - \E^x e^{-\lambda T_B} (u^\lambda (-X_{T_B})) \nonumber\\
												= & K^\lambda(-x) - \E^x e^{- \lambda T_B}K^\lambda \left( -X_{T_B} \right)+G^\lambda_{B^c}(x, 0)\label{lem:gtw_1}\\
												= & K^\lambda(-x) - \E^x e^{- \lambda T_B}K^\lambda \left( -X_{T_B} \right).\nonumber
\end{align}
Since $K^\lambda$ is bounded by $H$ and, by Proposition \ref{prop:1}, $H$ is bounded on $B$  because of its compactness, using the dominated convergence theorem, Lemma \ref{asym} and Corollary \ref{istnienie-K} we infer that
\begin{align*}
\lim_{\lambda \to 0^+} (\Re\psi)^{-1} (\lambda) \mathcal{L} (\P^x(T_B > \cdot))(\lambda) =  \big(K(-x) - \E^x K(-X_{T_B})\big)/C(\alpha, C_I).
\end{align*}
Let $ U(s) :=  \int_0^s \P^x(T_B > t)dt$. We have
\begin{align*}
\mathcal{L} U(\lambda)  = \frac {1}{\lambda} \mathcal{L} (\P^x(T_B > \cdot))(\lambda).
\end{align*}
Hence,
\begin{align*}
\lim_{\lambda \to 0^+} \lambda (\Re\psi)^{-1} (\lambda) \mathcal{L} U(\lambda) = \big(K(-x) - \E^x K (-X_{T_B})\big)/C(\alpha,C_I).
\end{align*}
Notice that $(\Re\psi)^{-1}$ is regularly varying with an exponent $1/\alpha$, thus, by the Tauberian theorem \cite[Theorem 1.7.1]{RegularVariation89} we can observe that
\begin{align*}
\lim_{t \to \infty} (\Re\psi)^{-1} (1/t) U(t) = \frac{1}{C(\alpha,C_I)\Gamma(1 + 1/\alpha)}\big(K(-x) - \E^x K (-X_{T_B})\big).
\end{align*}
Eventually, by the monotone density theorem \cite[Theorem 1.7.2]{RegularVariation89},
\begin{align*}
\lim_{t \to \infty} t(\Re\psi)^{-1} (1/t) \P^x(T_B > t) = \frac{1}{\alpha C(\alpha,C_I)\Gamma(1 + 1/\alpha)}\big(K(-x) - \E^x K(-X_{T_B})\big).
\end{align*}
\end{proof}
 Since $\E^x K(X_{T_0}) = 0$, we obtain the following Corollary.
\begin{corollary} \label{wgtw}
Assume that $\Re\psi$ varies regularly at $0$ with an exponent $\alpha\in(1,2]$ and suppose that $\lim_{\xi \to 0^+} \Im \psi(\xi)/\Re\psi(\xi)= C_I$. Then, for $x \in \R$,
\begin{align}\label{eq:4}
\lim_{t \to \infty} t(\Re\psi)^{-1} (1/t) \P^x(T_0 > t) =  \frac{1}{C(\alpha,C_I)\Gamma(1/\alpha)}K(-x).
\end{align}
\end{corollary}
Using Theorem \ref{thm:Re_asymp_Im} we conclude the asymptotic behaviour for L\'{e}vy measures of specific type \eqref{gest-miar}.
\begin{corollary}\label{cor:asymp_spec}
	Suppose that $\nu(dx)$ is of the form \eqref{gest-miar}. Assume that $\E X_1=0$ and $\Re \psi$ is regularly varying at $0$ with parameter $\alpha \in (1,2)$. Then \eqref{eq:4} holds true. In particular, this is the case for spectrally one-side L\'{e}vy processes.
\end{corollary}
\begin{proposition}
Suppose that $1/(1+\Re\psi  )$ is integrable and $\E X_1$ exists. If $\E X_1 \neq 0$, then
\[
\P^x(T_B>t)\cong \big(K(-x) - \E^x K(-X_{T_B})\big)\kappa, \quad t\to \infty.
\]
\end{proposition}
\begin{proof}
Let us observe that by \cite[Theorem 36.7]{Sato99}, we have $\kappa>0$. Hence, by \eqref{lem:gtw_1} and Corollary \ref{cor:K_moment} we obtain the claim.
\end{proof}

\begin{corollary}
Under the assumptions of the above proposition the compensated kernel exists and it is coharmonic.
\end{corollary}
\section{Harnack inequality and boundary behaviour}\label{sec:harmonic}
%\subsection{Harnack inequality ???}
\begin{lemma}\label{lem:1}
	Suppose that $\E X_1=0$ and $\Re \psi \in \WLSC{\alpha}{\chi}$ for some $\alpha >1$ and $\chi \in (0,1]$. Then there are $\delta_1 \in (0,1]$ and $c > 0$, depending only on the scalings, such that for any $R>0$,
	\[
	G_{(-R,R)}(x,y) \geq c H(R), \quad |x|,|y| \leq \delta_1 R.
	\]
\end{lemma}

\begin{proof}
	By the sweeping formula, for any $\lambda > 0$ and any $x,y \in \R$ we have
	\[
	G_{(-R,R)}(x,y) \geq G^{\lambda}(x,y) = U^{\lambda}(y-x) - \E^x e^{-\lambda\tauRR} U^{\lambda}\Big(y-X_{\tauRR}\Big).
	\]
	By \cite[Lemma 2.10]{GS2019}, there is $c_1 \in (0,1]$ such that $t\big\lvert b_{h^{-1}(1/t)}\big\rvert \leq c_1h^{-1}(1/t)$ for all $t>0$. Hence, by \cite[Theorem 5.4]{GS2019} with $\theta = (2+c_1)h^{-1}(\lambda)$, there is $c_2 \in (0,1]$ such that for all $|x|,|y|<h^{-1}(\lambda)$,
	\[
	U^{\lambda}(y-x) \geq \int_{1/\lambda}^{\infty}e^{-\lambda t}p(t,x)\,dt \geq c_2 \int_{1/\lambda}^{\infty} e^{-\lambda t} \frac{dt}{h^{-1}(1/t)}. 
	\]
	By \cite[Lemma 2.3]{GS2019}, there is $c_3 \in (0,1]$ such that
	\begin{equation*}
	U^{\lambda}(y-x) \geq \frac{c_2c_3}{\lambda h^{-1}(\lambda)}\int_1^{\infty} e^{-s} \frac{ds}{s^{1/\alpha}} = c_4 \frac{1}{\lambda h^{-1}(\lambda)},
	\end{equation*}
	with $c_4 = c_2c_3/\int_1^{\infty}e^{-s} s^{-1/\alpha} \,ds$.
	
	Next, using the estimate on the supremum of the density $p(t,\cdot)$ (see \cite[Theorem 3.1]{GS2019}) we infer that
	\[
	\E^x e^{-\lambda\tauRR} U^{\lambda}\Big(y-X_{\tauRR}\Big) \leq \E^x e^{-\lambda \tauRR}  \int_0^{\infty} e^{-\lambda t} \frac{dt}{h^{-1}(1/t)}.
	\]
	By the scaling property of $h^{-1}$,
	\begin{equation*}
	\int_0^{1/\lambda} e^{-\lambda t} \frac{dt}{h^{-1}(1/t)} \leq c^{-1} \frac{1}{\lambda^{1/\alpha}h^{-1}(\lambda)} \int_0^{1/\lambda} \frac{dt}{t^{1/\alpha}} = c_{\alpha} \frac{1}{\lambda h^{-1}(\lambda)}.
	\end{equation*}
	Moreover, by monotonicity of $h^{-1}$,
	\begin{equation*}
	\int_{1/\lambda}^{\infty} e^{-\lambda t} \frac{dt}{h^{-1}(1/t)} \leq \frac{1}{h^{-1}(\lambda)} \int_{1/\lambda}^{\infty} e^{-\lambda t} \,dt = e^{-1} \frac{1}{\lambda h^{-1}(\lambda)}.
	\end{equation*}
	Now let $t_0>0$. By Pruitt's estimate \cite{Pruitt81}, there is $c_5>0$ such that
	\begin{equation*}
	\E^x \Big[ \tauRR\leq t_0;\,e^{-\lambda \tauRR} \Big] \leq c_5 t_0 \big(h(R) + R^{-1}|b_R|\big).
	\end{equation*}
	Furthermore,
	\begin{equation*}
	\E^x \Big[\tauRR \geq t_0;\,e^{-\lambda \tauRR} \Big] \leq \frac{c_5e^{-\lambda t_0 }}{t_0 \big( h(R)+R^{-1}|b_R| \big)}.
	\end{equation*}
	Thus, if we set $t_0 = c_6/(h(R)+R^{-1}b_R)$ and $\lambda = c_7 \big( h(R)+R^{-1}b_R \big)$, where $c_6$ and $c_7$ are such that $c_6 \leq c_4/(4c_{\alpha}+4e^{-1})$ and $e^{-c_7} \leq c_4c_6/(4c_5(c_{\alpha}+e^{-1}))$, then putting everything together yields
	\[
	G_{(-R,R)}(x,y) \geq \frac{c_4}{2}\frac{1}{\lambda h^{-1}(\lambda)}.
	\]
	Since by \cite[Lemma 2.10]{GS2019} we have $\lambda \approx h(R)$, using scaling properties of $h^{-1}$ we get that
	\[
	G_{(-R,R)}(x,y) \gtrsim \frac{1}{Rh(R)}, \quad |x|,|y| \leq \delta_1R,
	\]
	with some $\delta_1 \in (0,1]$, and the claim follows by Proposition \ref{prop:1}.
\end{proof}

\begin{proposition}\label{prop:4}
	Suppose $\Re \psi \in \WLSC{\alpha}{\chi}$ for some $\alpha>1$ and $\chi \in (0,1]$. There is $\delta_2 \leq \delta_1$ dependent only on the scalings such that for any $R>0$ any non-empty $A \subset (-\delta_2R,\delta_2R)$,
	\[
	\P^x\Big( T_A> \tauRR \Big) \geq \frac12, \quad |x|\leq \delta_2R.
	\]
\end{proposition}
\begin{proof}
	Let $|a|\leq R/4$ and $D=(-R/2,0)\cup(0,R/2)$. By \cite[Lemma 3]{TG14} and Proposition \ref{prop:3}, there is $C_1>0$ such that for $|x-a|\leq R/4$,
	\[
	\P^x \Big( T_a>\tauRR \Big) \leq \P^{x-a} \Big( T_0 > \tau_{(-R/2,R/2)} \Big) \leq C_1 h(R/2) \E^{x-a}\tauD \leq 8C_1 Rh(R)H(x-a).
	\]
	In view of Proposition \ref{prop:1}, there is $C_2>0$ dependent only on the scalings such that
	\[
	\P^x \Big( T_a>\tauRR \Big) \leq C_2 \frac{H(x-a)}{H(R)}, \quad |x-a|<R/4.
	\]
	Since by Proposition \ref{prop:1}, $H \in \WLSC{\alpha-1}{\tilde{\chi}}$ for some $\tilde{\chi} \in (0,1]$, we can pick $\delta_2 < 1/2$ such that
	\[
	\P^x \Big( T_a > \tauRR \Big) \leq \frac12, \quad |x-a|<2\delta_2R.
	\]
	It follows that if $x \in A \subset (-\delta_2R,\delta_2R)$ and $a \in A$, then
	\[
	\P^x \Big( T_A \geq \tauRR \Big) \leq \P^x \Big( T_a > \tauRR \Big) \leq \frac12,
	\]
	and the proof is completed.
\end{proof}
Denote $R_0 = \delta_2R$, where $\delta_2$ is taken from Proposition \ref{prop:4}.
\begin{proposition}
	Suppose $\Re \psi \in \WLSC{\alpha}{\chi}$ for some $\alpha >1$ and $\chi \in (0,1]$. Then for any $R>0$ and any non-negative function $F$ such that $(\supp F)^c \subset (-R,R)$,
	\[
	\E^x F \Big( X_{\tau_{(-R_0,R_0)}} \Big) \leq \frac{1}{c}\E^y F \Big( X_{\tauRR} \Big), \quad |x|,|y| \leq R_0.
	\]
\end{proposition}
\begin{proof}
	Let us denote, for any $w \in \R$ and a Borel set $A$, $\nu(w,A)=\nu(A-w)$. By the Ikeda-Watanabe formula and Lemma \ref{lem:1},
	\begin{align*}
	\E^y F \Big( X_{\tauRR} \Big) &\geq \int_{(-R,R)^c} \int_{-R_0}^{R_0} F(z) G_{(-R,R)}(y,w) \,\nu(w,dz) \,dw \\ &\geq cH(R)\int_{(-R,R)^c}\int_{-R_0}^{R_0} F(z) \,\nu(w,dz)\,dw.
	\end{align*}
	On the other hand, by the Ikeda-Watanabe formula, Proposition \ref{prop:2} and subadditivity of $H$,
	\begin{align*}
	\E^x F \Big( X_{\tau_{(-R_0.R_0)}} \Big) &\leq \int_{(-R,R)^c} \int_{-R_0}^{R_0} F(z) \Go(x+R_0,y+R_0)\,\nu(w,dz) \,dw \\ &\leq H(R_0) \int_{(-R,R)^c}\int_{-R_0}^{R_0} F(z) \,\nu(w,dz) \,dw.
	\end{align*}
	Hence,
	\[
	\E^x F \Big( X_{\tau_{(-R_0,R_0)}} \Big) \leq \frac{1}{c}\E^y F \Big( X_{\tauRR} \Big).
	\]
\end{proof}
\begin{theorem}\label{thm:harnack}
	Suppose $\Re \psi \in \WLSC{\alpha}{\chi}$ for some $\alpha>1$ and $\chi\in(0,1]$. Then the global scale invariant Harnack inequality holds, i.e. there is a constant $C_H$ dependent only on the scalings such that for any $R>0$ and any non-negative harmonic function on $(-R,R)$ we have
	\begin{equation}\label{eq:harnack_ineq}
	\sup_{x \in (-R/2,R/2)} h(x) \leq C_H \inf_{x \in (-R/2,R/2)} h(x).
	\end{equation}
\end{theorem}
\begin{proof}
	Suppose first that $h$ is bounded. Then using the approach of Bass and Levin \cite{BL02} we infer that there exist constants $c_1 = c_1(\alpha,\chi)$ and $a = a(\alpha,\chi) \in (0,1]$ such that for any non-negative, bounded and harmonic function on $(-R,R)$,
	\[
	\sup_{x \in (-aR,aR)}h(x) \leq c_1 \inf_{x \in (-aR,aR)} h(x).
	\]
	Now we apply a standard chain argument to get
	\[
	\sup_{x \in (-R/2,R/2)} h(x) \leq C_H \inf_{x \in (-R/2,R/2)} h(x).
	\]
	It remains to observe that the boundedness assumption on $h$ may be removed in the similar way as in the proof of \cite[Theorem 2.4]{SV02}.
\end{proof}
Thanks to the Harnack property we are able to prove a relation between renewal functions and their derivatives, and provide a sharp estimate for the Green function of the positive half-line.  
\begin{corollary}\label{cor:V_approx_V'}
	Suppose that $\E X_1=0$ and $\Re \psi \in \WLSC{\alpha}{\chi}$ for some $\alpha>1$ and $\chi \in (0,1]$. Then there is $c \geq 1$ such that for all $x>0$,
	\[
	c^{-1} \frac{V(x)}{x} \leq V'(x) \leq c \frac{V(x)}{x},
	\]
	and
	\[
	c^{-1} \frac{\widehat{V}(x)}{x} \leq \widehat{V}'(x) \leq c \frac{\widehat{V}(x)}{x},
	\]
	In particular, $V',\widehat{V}' \in \WLSC{\alpha-2}{\tilde{\gamma}}$ for some $\tilde{\gamma} \in (0,1]$. 
\end{corollary}
\begin{proof}
	First, let us consider the second part of the claim. Let $x>0$. Recall that $\widehat{V}'$ is harmonic on $(0,\infty)$. Thus, by Theorem \ref{thm:harnack},
	\[
	\widehat{V}(x) \geq \int_{x/2}^x \widehat{V}'(s)\,ds \geq \frac{1}{2C_H}x\widehat{V}'(x).
	\]
	On the other hand, since $\Re \psi$ is the same for $\mathbf{X}$ and $\widehat{\mathbf{X}}$, we may apply \cite[Lemma 8]{TG19} for $\widehat{V}$. Let $c_1$ be taken from \cite[Lemma 8]{TG19} and $\delta \in \Big(0,(c_1/2)^{1/(\alpha-1)}\Big]$. Then, again by Theorem \ref{thm:harnack},
	\[
	C_H(1-\delta)x\widehat{V}'(x) \geq \int_{\delta x}^{x} \widehat{V}'(s)\,ds =\widehat{V}(x)-\widehat{V}(\delta x) \geq 1-c_1^{-1}\delta^{\alpha-1} \widehat{V}(x) \geq \frac12 \widehat{V}(x).
	\]
	Now, the lower scaling property follows immediately by \cite[Lemma 8]{TG19}. For the proof of the first part it remains to observe that by the previous remark on the real part of the characteristic exponent, $V$ also satisfies the Harnack inequality (with the same constant) and one can repeat the reasoning above to finish the proof.
\end{proof}
\begin{corollary}\label{cor:green_0inf}
	Suppose that $\E X_1=0$ and $\Re \psi \in \WLSC{\alpha}{\chi}$ for some $\alpha>1$ and $\chi \in (0,1]$. Then
	\[
	G_{(0,\infty)}(x,y) \approx \left\{\begin{array}{lr}\widehat{V}(x)V'(y),& 0<x\leq y,\\
\widehat{V}'(x)V(y),& 0<y<x.\end{array}\right.
	\]
	The comparability constant depends only on the scaling characteristics.
\end{corollary}
\begin{proof}First assume that $0<x\leq y$.
	Recall that
	\[
	G_{(0,\infty)}(x,y) = \int_0^x \widehat{V}'(u)V'(y-x+u)\,du, \quad 0<x\leq y.
	\]
	Since $V$ is monotone and subadditive, for any $\lambda \geq 1$ and $x>0$ we have
	\begin{equation}\label{eq:5}
	V(\lambda x) \leq 2 \lambda V(x).
	\end{equation}
	That, in view of \cite[Lemma 8]{TG19}, implies that $V'$ is almost decreasing, and consequently,
	\[
	G_{(0,\infty)}(x,y) \gtrsim \int_0^x \widehat{V}'(u)V'(y)\,du = \widehat{V}(x)V'(y).
	\]
	Next, let $x<y<2x$. By Corollary \ref{cor:V_approx_V'}, \cite[Corollary 5]{TG19}, and almost monotonicity of $V'$,
	\[
	G_{(0,\infty)}(x,y) \lesssim \int_0^x \widehat{V}'(u)V'(u)\,du \approx \int_0^x \frac{du}{u^2h(u)}.
	\]
	Using scaling property of $h$, \cite[Corollary 5]{TG19} and Corollary \ref{cor:V_approx_V'}, we conclude that
	\[
	G_{(0,\infty)}(x,y) \lesssim \frac{1}{xh(x)} \approx \frac{\widehat{V}(x)V(x)}{x} \lesssim \frac{\widehat{V}(x)V(y)}{y} \lesssim \widehat{V}(x)V'(y).
	\]
	where the third inequality follows from \eqref{eq:5}. Finally, for $y \geq 2x$ we use scaling property of $V'$ with index $\alpha-2$ (Corollary \ref{cor:V_approx_V'}) to obtain that
	\[
	V'(y-x+u) \lesssim V'(y) \bigg( \frac{y}{y-x+u} \bigg)^{2-\alpha} \leq 2^{2-\alpha}V'(y),
	\]
	and the first part follows. 
	
	If $0<y\leq x$ we use the Green function for the dual process to get the claim.
\end{proof}

Assume that $\E X_1=0$ and $\Re \psi \in \WLSC{\alpha}{\chi}$ for some $\alpha>1$ and $\chi \in (0,1]$. By \cite[Theorem 1]{Silverstein80}, $V'$ is coharmonic on $(0,\infty)$, that is harmonic on $(0,\infty)$ for the dual process $\widehat{\mathbf{X}}$. Since $\Re \psi$ is symmetric, the Harnack inequality for $\widehat{\mathbf{X}}$ holds as well. Thus, by Theorem \ref{thm:harnack}, for any $0<\delta \leq w \leq u \leq w+2\delta$,
\begin{equation}\label{eq:H_prop}
V'(u) \leq C_H V'(w).
\end{equation}
Proofs of the remaining lemmas of this Sections follow directly results obtained in \cite[Subsection 4.2]{GR17} and therefore they are omitted. 	
\begin{lemma}\label{lem:BHI_upper}
	Suppose that $\E X_1=0$ and $\Re \psi \in \WLSC{\alpha}{\chi}$ for some $\alpha>1$ and $\chi \in (0,1]$. Let $F(z)$ be non-negative, $F(x)\leq F_1(x)$ on $\R$ and $F(x+y)\leq F_1(x)+F_1(y)$, for $x,y\in\R$ . Suppose that $\E^x F \big(X_{\tauoinf} \big) \leq F(x)$ and $\E^x F_1 \big(X_{\tauoinf} \big) \leq F_1(x)$ for $x>0$. Then there is $c>0$ such that for any $0<x<1$,
	\[
	\E^x \Big[X_{\tauoinf}\leq -2 ;\,  F \big( X_{\tauoinf} \big) \Big] \leq cC_H^2 F_1^*(1) \frac{\widehat{V}(x)}{\widehat{V}(1)}.
	\]
	The constant $c$ depends only on the scalings.
\end{lemma}
\begin{proof}
	Follows directly by proof of \cite[Lemma 4.7]{GR17} with applications of Lemma 4.6 and Lemma 2.9 replaced by Corollary \ref{cor:green_0inf} and \cite[Corollary 5]{TG19}, respectively, and using a function $F_1$ instead of subadditivity of $F$.
\end{proof}

\begin{lemma}\label{lem:BHI_lower}
	Suppose $\E X_1=0$ and $\Re \psi \in \WLSC{\alpha}{\chi}$ for some $\alpha>1$ and $\chi \in (0,1]$. Let $F$ be a non-negative harmonic function on $(0,2R)$ for some $R>0$. Suppose that $r>0$ is such that $\widehat{V}(R) \geq 2\widehat{V}(r)/\tilde{c}$, where $c$ is taken from Proposition \ref{prop:prob_exit_sharp}. Then for $0<x<r$,
	\[
	\frac{F(x)}{F(r)} \geq \frac{c}{4}C_H^{-1-R/r} \frac{\widehat{V}(x)}{\widehat{V}(r)},
	\]
	where $C_H$ is the constant from the Harnack inequality \eqref{eq:harnack_ineq}.
\end{lemma}
\begin{proof}
	Follows directly the proof of \cite[Lemma 4.8]{GR17} with applications of Theorem 4.5 and Lemma 2.11 replaced by Theorem \ref{thm:harnack} and Proposition \ref{prop:prob_exit_sharp}, respectively.
\end{proof}

\section{Estimates}\label{sec:estimates}
In this Section we prove sharp two-sided estimates on the tail of the first hitting time of the interval. Our main result here is Theorem \ref{thm:main_hitting_time}. We also provide an analogous estimate for the specific case of spectrally negative L\'{e}vy processes. Afterwards, in Subsection \ref{subsec:ex} we point out a large class of non-symmetric L\'{e}vy processes which satisfy its assumptions.

We begin with the following estimate on $u^{\lambda}$.
\begin{lemma}\label{lem:oszacowanie-u}
Assume that there exist constants $a>0$, $b\geq 0$ such that $ |\Im \psi(\xi)|\leq b \Re \psi (\xi)$,  $\xi\in\R$ and $a\psi^*(x)\leq \Re\psi(x)$, $x\geq 0$. 
%Denote $B_\lambda := \frac{\int_{\Rel\psi(\xi)^{-1}(\lambda)}^1 \frac{d\xi}{\Rel\psi(\xi)}}{\int_{\Rel\psi(\xi)^{-1}(\lambda)}^\infty \frac{d\xi}{\Rel\psi(\xi)}}$.
Then we have
\begin{align*}
 \frac{a}{4(1+b^2)}H \left( \frac{1}{(\Re\psi)^{-1}(\lambda)}\right)\leq u^\lambda(0) \leq \frac{2\pi^2(1+b^2)}{a} H \left( \frac{1}{(\Re\psi)^{-1}(\lambda)}   \right) .
\end{align*}
%If there exists $a$ such that $\Rel \psi(x) \geq a \Rel \psi^*(x)$ then
%\begin{align*}
%u^\lambda(0) \geq \frac{B_\lambda a}{8(1+C^2)} K_R \left( \frac{1}{\Rel\psi^{-1}(\lambda)}   \right) .
%\end{align*}
\end{lemma}

\begin{proof}
%Denote by $\theta(\xi) := \Rel \psi (\xi), \ \theta^* (\xi) := \Rel \psi^*(x)$ i $\omega(\xi) := \Ima \psi (\xi)$. By (\ref{postac-u}) one can observe
Since $ |\Im \psi(\xi)|\leq b \Re \psi (\xi)$ we have
\begin{align*}
 \frac{1}{\pi(1+b^2)} \int_0^\infty \frac {d\xi}{\lambda + \Re\psi (\xi)}\leq u^\lambda (0) & \leq \frac{1}{\pi} \int_0^\infty \frac {d\xi}{\lambda + \Re\psi (\xi)}. %\geq  \frac{1}{4 \pi}\int_{0}^1 \frac {\lambda + \theta (\xi)}{(\lambda + \theta(\xi))^2 + (C\theta (\xi))^2}d\xi \\
	              %& \geq \frac{1}{(1+C^2)4 \pi} \int_{0}^1 \frac {d\xi}{\lambda + \theta^* (\xi)} \geq \frac{1}{(1+C^2)4 \pi}B_\lambda \left(\int_0^{\theta^{-1}(\lambda)} \frac{1}{2\lambda} d\xi +  \int_{\theta^{-1}(\lambda)}^\infty \frac{d\xi}{2\theta^*(\xi)} \right) , %+  \frac{1}{2 \pi}\int_1^\infty \frac{1}{\theta^*(\xi)},
\end{align*}
Hence, by \cite[Lemma 2.15]{GR17}, for $\lambda>0$,
\[
\frac{a}{4\pi(1+b^2)}\int^\infty_0(1-\cos(s/(\Re\psi)^{-1}(\lambda)))\frac{ds}{\Re\psi(s)}\leq u^\lambda (0)\leq \frac{3\pi}{2a}\int^\infty_0(1-\cos(s/(\Re\psi)^{-1}(\lambda)))\frac{ds}{\Re\psi(s)}.
\]
Using $ |\Im \psi(\xi)|\leq b \Re \psi (\xi)$ we infer that
$$\pi H(x)\leq \int^\infty_0(1-\cos(xs))\frac{ds}{\Re\psi(s)}\leq (1+b^2)\pi H(x),$$
which ends the proof.
\end{proof}

\begin{lemma}\label{lem:Klambda_large}
	Suppose that $\E X_1=0$ and  $\Re\psi \in \WLSC{\alpha}{\chi}$ for some $\alpha>1$ and $\chi\in (0,1]$. Then there exists $c=c(\alpha,\chi)$ such that, for any $a,x>0$,
	\[
	K^\lambda(x)\geq c(1-e^{-a})u^\lambda(0),\quad \lambda\geq ah(x).
	\]
\end{lemma}
\begin{proof}
Since $$K^\lambda(x)=\lambda u^\lambda(0)\mathcal{L}[\P^x(T_0>\cdot)](\lambda),$$
it is enough to prove that $\mathcal{L}[\P^x(T_0>\cdot)](\lambda)\geq c/\lambda$, if $\lambda\geq a h(x)$. Using estimates of the tail distribution of the first exit time from the positive half-line \cite[Theorem 6]{TG19} we conclude that there is $c_1$ such that
\begin{align*}
\mathcal{L}[\P^x(T_0>\cdot)](\lambda)&\geq c_1 \int^\infty_0\left(1\wedge \frac{\hat{V}(x)}{\hat{V}\big(h^{-1}(1/s)\big)}\right)e^{-\lambda s}\,ds\geq c_1\int^{1/h(x)}_0e^{-\lambda s}\,ds\geq c_1(1-e^{-a})\lambda^{-1}.
\end{align*}
\end{proof}
\begin{proposition}\label{prop:T0_upper}
Assume that there exist constants $a>0$, $b\geq 0$ such that $ |\Im \psi(\xi)|\leq b \Re \psi (\xi)$,  $\xi\in\R$ and $a\psi^*(x)\leq \Re\psi(x)$ for $x\geq 0$. Then
\begin{align*}
\P^x(T_0 > t) \leq \frac{4(e-1)}{e} \frac{(1+b^2)}{a}\frac{H(x)}{H\big(1/(\Re \psi)^{-1}(1/t)\big)}  \wedge 1.
\end{align*}
\end{proposition}

\begin{proof}
Recall that
\begin{align*}
\lambda\mathcal{L}\left( \P^x(T_0 > \cdot)  \right)(\lambda) = \left[ 1 - \E^x e^{-\lambda T_0} \right] =  \frac{u^\lambda(0) - u^\lambda(-x)}{u^\lambda(0)} =  \frac{K^\lambda(-x)}{u^\lambda(0)}.
\end{align*}
By Lemma \ref{lem:oszacowanie-u},
\begin{align*}
\mathcal{L}\left( \P^x(T_0 > \cdot)  \right)(\lambda)  \leq \frac{4(1+b^2)}{a} \frac{H(x)}{H\big(1/(\Re \psi)^{-1}(\lambda)\big)}.
\end{align*}
Therefore, using \cite[Lemma 5]{BGR14} we conclude that
\begin{align*}
\P^x(T_0 > t) \leq  \frac{e}{e-1} \frac{4(1+b^2)}{a}  \frac{H(x)}{H\big(1/(\Re \psi)^{-1}(1/t)\big)}. 
\end{align*}
\end{proof}

\begin{corollary}\label{cor:T0_upper_scalings}
	Assume that $\E X_1=0$ and $\Re \psi \in \WLSC{\alpha}{\chi}$ for some $\alpha>1$ and $\chi \in (0,1]$. Then there is $c>0$ such that for all $t>0$,
	\[
	\P^x(T_0 > t) \leq c \frac{H(x)}{H\big(h^{-1}(1/t)\big)}  \wedge 1.
	\]
	The constant $c$ depends only on the scalings.
\end{corollary}
\begin{proof}
	Using \cite[Lemma 12]{TG19} and \cite[Remark 3.2]{GS2019} we see that the assumptions of Proposition \ref{prop:T0_upper} are satisfied. Now it remains to apply comparability of $1/(\Re \psi)^{-1}$ and $h^{-1}$ together with Proposition \ref{prop:1}.
\end{proof}

\begin{lemma}\label{lem:upperB_small} 
	Suppose $\E X_1=0$ and $\Re\psi\in\WLSC{\alpha}{\chi}$ for some $\alpha>1$ and $\chi \in (0,1]$. If $x>1$ and $ t< 1/h(1)$ then
	\[
	\P^x(T_{B_1}>t)\approx  \frac{\hat{V}(x-1) } {\hat{V}\big(h^{-1}(1/t)\big)}\wedge 1.
	\]
	The comparability constant depends only on the scalings.
\end{lemma}
\begin{proof}
	Of course, the lower bound is a consequence of the estimates on the tail for the first exit time from a half-line, that is \cite[Theorem 6]{TG19}. By subadditivity of $\hat{V}$, it is enough to consider $1<x<1+h^{-1}(1/t)/2$, because if $x$ is larger, by the lower bound the probability is comparable to $1$.  

	To prove the estimate from the above let us denote  $r=h^{-1}(1/t)$. Notice that $r<1$ and we have 
	\[
	\P^x(T_{B_1}>t)\leq \P^x\big(\tau_{(1,1+r)}>t\big) + \P^x\big(\big|X_{\tau_{(1,1+r)}}-1\big|>r\big).
	\]
	Combining \cite[Lemma 3]{TG14} and \cite[the proof of Corollary 3]{TG19} we obtain
	\[\P^x(|X_{\tau_{(1,1+r)}}-1|>r)\leq c\E^x\tau_{(1,1+r)}h(r) \leq c\hat{V}(x-1)V(r)h(r),
	\]
	for some $c>0$, where in the last inequality we used \cite[Proposition 4]{TG19}. Finally by \cite[Corollary 5]{TG19},
	\[
	\P^x(|X_{\tau_{(1,1+r)}}-1|>r)\leq c\frac{\hat{V}(x-1) } {\hat{V}(r)}.
	\]
	This together with \cite[Theorem 6]{TG19} imply
	\[
	\P^x(T_{B_1}>t)\leq c \frac{\hat{V}(x-1) } {\hat{V}(r)}.
	\]
\end{proof}

\begin{lemma}\label{lem:upperB_large} 
	Assume that $\E X_1=0$ and $\Re\psi \in \WLSC{\alpha}{\chi}$ for some $\alpha>1$ and $\chi \in (0,1]$. If $x>1$ and $ t\geq 1/h(1)$ then
	\[
	\P^x(T_{B_1}>t)\le  c\frac{\hat{V}(x-1)} {\hat{V}(x)}\frac{ H(x)} { H\big(h^{-1}(1/t)\big)}\wedge 1 \approx  \frac{\hat{V}(x-1) } {\hat{V}(x)}\frac{H(x)}{t/h^{-1}(1/t)}\wedge 1.
	\]
	The constant $c$ depends only on the scalings.
\end{lemma}
\begin{proof} 
	If $x \ge2$ we have, by subadditivity and monotonicity of $\hat{V}$,  $\frac{\hat{V}(x-1)}{\hat{V}(x)}\ge \frac12$, hence the claim  follows from \cite[Lemma 12]{TG19} and  Corollary \ref{cor:T0_upper_scalings}.

	Let  $1<x<2$. By \cite[Theorem 6]{TG19},
	\[
	\P^x(\tau_{(1,\infty)}>t)\approx 1\wedge\frac{\hat{V}(x-1)}{\hat{V}\big(h^{-1}(1/t)\big)}.
	\]

	Since $t> 1/h(1)$, using  subadditivity of $\hat{V}$  and \cite[Lemma 2.1]{GS2019} we obtain %.  Then,  by Lemma \ref{halfspace} and subadditivity of $V$, 
	\begin{align*}
	\P^x(\tau_{(1,\infty)}>t/2) &\leq c_1\frac{\hat{V}(x-1)}{\hat{V}(1)} \frac{\hat{V}(1)}{\hat{V}\big(h^{-1}(1/t)\big)}\leq c_2\frac{\hat{V}(x-1)}{\hat{V}(1)}\P^2(\tau_{(1,\infty)}>t) \\ &\leq c_3\frac{\hat{V}(x-1)}{\hat{V}({x})}{\P^1(T_{0}>t)}.
	\end{align*}
	Since
	\[
	\P^x(T_{B_1}>t)\le  \P^x(\tau_{(1,\infty)}>t/2)+   \E^x\P^{X_{\tau_{(1,\infty)}}}(T_{B_1}>t/2),
	\]
	due to Proposition \ref{prop:1} and Corollary \ref{cor:T0_upper_scalings}, it is enough to estimate the second term. We have
	\begin{align*}
	\E^x\P^{X_{\tau_{(1,\infty)}}}(T_{B_1}>t/2)&\le \E^x[X_{\tau_{(1,\infty)}}\leq-1; \P^{X_{\tau_{(1,\infty)}}}(T_{1}>t/2)]\\&=\E^{x-1}[X_{\tau_{(0,\infty)}}\leq-2;\P^{X_{\tau_{(0,\infty)}}}(T_{0}>t/2)].
	\end{align*}
	Let $F(z)= \P^{z}(T_{0}>t/2)$. Observe that
	\begin{align*}
	F(z) &= \P^z(\tau_{(0,\infty)}>t/2)+\E^z\Big[\tau_{(0,\infty)}\leq t/2;\, \P^{X_{\tau_{(0,\infty)}}}\big(T_{B_1}>t/2-\tau_{(0,\infty)}\big) \Big] \\ &\geq \E^z\Big[ \P^{X_{\tau_{(0,\infty)}}}\big(T_{B_1}>t/2\big)\Big] \\ &=\E^zF \big(X_{\tau_{(0,\infty)}}\big).
	\end{align*}
	Furthermore,
	\[
	F(x+y)\leq \P^{x+y}(T_x>t/4)+ \E^x\Big[ T_x\leq t/4;\, \P^{X_{T_x}}(T_0>t/4)\Big] \leq \P^{y}(T_0>t/4) + \P^{x}(T_0>t/4).
	\]
	Hence, $F$ and $F_1(z)= \P^{z}(T_{0}>t/4)$ satisfy the assumptions of Lemma \ref{lem:BHI_upper}. Therefore, the conclusion follows from Lemma \ref{lem:BHI_upper} and Proposition \ref{prop:T0_upper}.
\end{proof}

\begin{lemma}\label{lem:ball20} 
	Assume $\E X_1=0$ and $\Re\psi\in \WLSC{\alpha}{\chi}$ for some $\alpha>1$ and $\chi \in (0,1]$. If $x_0>1$, $1<x\leq x_0$ and $ t> 1/h(1)$, then  there is $c=c(x_0,\alpha,\chi)>0$ such that
	\[
	\P^x(T_{B_1}>t)\geq  c\frac{\hat{V}(x-1)}{\hat{V}(x_0)} \P^{x_0}(T_0>2t).
	\]
\end{lemma}

\begin{proof}
	With Lemma \ref{lem:BHI_lower} and \cite[Theorem 6]{TG19} at hand, the proof is the same as the first part of the proof of \cite[Lemma 5.4]{GR17} and therefore it is omitted.
\end{proof}

\begin{lemma}\label{lem:point_lower} 
	Assume that $\E X_1=0$ and $\Re\psi\in \WLSC{\alpha}{\chi}$ for some $\alpha>1$ and $\chi \in (0,1]$. Suppose that there exist  constants $c>0$ and $a>0$ such that for $x>0$, $K^\lambda(x)\geq c H(x)$, $\lambda\leq a h(x)$. Then and $ t> 1/h(1)$,   there is $\tilde{c}=\tilde{c}(x_0,\alpha,c)>0$ such that
	\[
	\P^x(T_0>t)\geq  \tilde{c} \bigg( 1\wedge\frac{H(x)}{H\big(h^{-1}(1/t)\big)} \bigg),\quad x,t>0.
	\]
\end{lemma}
We remark that in case of symmetric L\'{e}vy processes the last assumption follows from \cite[Lemma 2.15]{GR17}.
\begin{proof}
	By Lemma \ref{lem:oszacowanie-u}, comparability $h$ and $\psi$ and Proposition \ref{prop:1},
	\[
	\lambda\mathcal{L}(\P^x(T_0>\cdot))(\lambda)\approx\frac{K^\lambda(x)}{  H\left(h^{-1}(\lambda)\right)}.
	\]

	Let $x>0$, $\lambda>0$ and $s>1$. Combining Lemma \ref{lem:Klambda_large} with the assumption on $K^\lambda$ we obtain, for $\lambda\geq ah(x)$ or $\lambda s\leq a h(x)$,
	\[
	\frac{K^{\lambda s}(x)}{K^\lambda(x)} \lesssim 1.
	\]
	If $\lambda s\geq a h(x)\geq \lambda $ we have, by Lemma \ref{lem:oszacowanie-u} 
	\[
	\frac{K^{\lambda s}(x)}{K^\lambda(x)}\approx \frac{u^\lambda(0)}{H(x)}\approx\frac{H\left(h^{-1}(\lambda s)\right)}{H(x)}\approx \frac{xh(x)}{\lambda s h^{-1}(\lambda s)}\leq c s^{-1} .
	\]
	Thus, 
	\[
	\frac{K^{\lambda s}(x)}{K^\lambda(x)}\leq c,
	\] 
	and consequently,
\begin{align*} 
\frac{\mathcal{L}(\P^x(T_0>\cdot))(\lambda s)}{\mathcal{L}(\P^x(T_0>\cdot))(\lambda)}\leq c\frac{\lambda}{\lambda s}\frac{K^{\lambda s}(x)}{K^\lambda(x)}\frac{H\left(h^{-1}(\lambda)\right)}{H\left(h^{-1}(\lambda s)\right)}\leq c \frac{h^{-1}(\lambda s)}{h^{-1}(\lambda)}\leq c s^{-1/2},
\end{align*}
where $c$ depends only on the scalings and $a$.
 Hence, by \cite[Lemma 13]{BGR14} there exists a constant $c_1$ that depends only on the scalings such that
$$\P^x(T_0>t) \geq c_1 \frac{K^{1/t}(x)}{ H\left(\frac{1}{\psi^{-1}(1/t)}\right)}.$$
 For $t>1/ah(x)$ we get the claim by the comparability $K^{1/t}$ with $K$ and for $t\leq 1/(ah(x))$ we use estimates for the positive half-line \cite[Theorem 6]{TG19}.
\end{proof}

\begin{proposition}\label{prop:ball} 
	Assume $\E X_1=0$ and $\Re\psi\in \WLSC{\alpha}{\chi}$ for some $\alpha>1$ and $\chi \in (0,1]$. Suppose that there exist constants $c>0$ and $a>0$ such that for $x>0$, $K^\lambda(x)\geq c H(x)$, $\lambda\leq a h(x)$. Then there is $x_0\geq 2$, which depends only on the scaling characteristics and $a$, such that for $x\geq x_0$ we have, for $t>1/h(1)$,
	\[
	\P^x(T_{B_1}>t)\geq \tilde{c}   \left(\frac{H(|x|)}{H(h^{-1}(1/t))}\wedge 1\right)\approx \left(\frac{H(|x|)}{t/h^{-1}(1/t)}\wedge 1\right).
	\]
	The constant $\tilde{c}$ depends only on the scalings and $a$.
\end{proposition}
The proof is very similar to the proof of \cite[Proposition 5.3]{GR17} with modifications like in the above proof therefore it is omitted.

We now proceed to the proof of the main result of this Section.
\begin{theorem}\label{thm:main_hitting_time}
	Suppose that $\E X_1=0$ and $\Re \psi \in \WLSC{\alpha}{\chi}$ for some $\alpha>1$ and $\chi>0$. Assume there exist constants $c_1>0$ and $a>0$ such that for $x>0$, $K^{\lambda}(x) \geq c_1 H(x)$, $\lambda \leq a h(x)$. Then for any $R>0$ and $x>R$,
	\[
	\P^x \big( T_{B_R} > t \big) \approx \frac{\widehat{V}(x-R)}{\widehat{V} \big( h^{-1}(1/t) \big)} \wedge 1, \quad t< 1/h(R),
	\]
	and
	\[
	\P^x \big( T_{B_R} > t \big) \approx \frac{\widehat{V}(x-1)}{\widehat{V}(x)} \frac{H(x)}{H \big( h^{-1}(1/t) \big)} \wedge 1 \approx \frac{\widehat{V}(x-1)}{\widehat{V}(x)} \frac{H(x)}{t/h^{-1}(1/t)} \wedge 1, \quad t \geq 1/h(R).
	\]
\end{theorem}
\begin{proof}
	The case $R=1$ follows by Lemma \ref{lem:upperB_small}, Lemma \ref{lem:upperB_large}, and Proposition \ref{prop:ball}. Now we may proceed as in the proof of \cite[Theorem 5.5]{GR17} to obtain the claim for any $R>0$.
\end{proof}

Let us turn our attention to the specific class of L\'{e}vy processes.
\begin{lemma}
	Suppose $\E X_1=0$, $\Re\psi\in \WLSC{\alpha}{\chi}$ for some $\alpha>1$, $\chi \in (0,1]$, and let $R\in[0,\infty)$. Assume that 
	\begin{equation*}
	\int^1_0\frac{\nu(y,\infty)}{h(y)}\frac{dy}{y}<\infty.
	\end{equation*}
	Then 
	\[
	\P^x\big(T_{B_R}>t\big)\approx \frac{H(x-R)}{H\big(h^{-1}(1/t)\big)}\wedge 1,\quad R<x<R+1,\, 0<t<1/h(1).
	\]
\end{lemma}
\begin{proof}
	A consequence of Proposition 14 and Corollary 5 in \cite{TG19}  is $\tilde{V}(x)\approx\frac{1}{xh(x)}$, $0<x\leq 1$. This together with Proposition \ref{prop:1} imply $\tilde{V}(x)\approx H(x)$, $0\leq x\leq 1$. Hence the claim holds due to \cite[Theorem 6]{TG19} and Corollary \ref{cor:T0_upper_scalings}.
\end{proof}
Similarly the consequence of \cite[Proposition 15 and Section 3]{TG19} is the following.
\begin{lemma}
	Suppose $\E X_1=0$ and $\Re\psi\in \WLSC{\alpha}{\chi}$ with $\alpha>1$, $\chi \in (0,1]$, and let $R\in[0,\infty)$. Assume that any of the following holds true: 
	\begin{itemize}
		\item[(i)] $\E X_1^2<\infty$,
		\item[(ii)] there is $C,r>0$ such that $\nu(x,\infty)\leq C \nu(-\infty,-x)$, $x>r$, and
		\begin{equation*}
		\int_1^\infty\frac{\nu(y,\infty)}{h(y)}\frac{dy}{y}<\infty.
		\end{equation*}
	\end{itemize}
	Then 
	\[
	\P^x\big(T_{B_R}>t\big)\approx \frac{H(x-R)}{H\big(h^{-1}(1/t)\big)}\wedge 1,\quad x\geq R+1,\, t>0.
	\]
\end{lemma}

Combining the above two lemmas and the fact that for spectrally negative processes $T_0=\tau_{(-\infty,0)}$ if the process starts from negative half-line, we obtain the following result.
\begin{corollary}\label{cor:est_spec_neg}
	Suppose $\E X_1=0$ and $\Re\psi\in \WLSC{\alpha}{\chi}$ with $\alpha>1$, $\chi \in (0,1]$, and let $R\in[0,\infty)$. Assume that $\mathbf{X}$ is spectrally negative, i.e. $\nu(0,\infty)=0$. Then  
	\[
	\P^x\big(T_{B_R}>t\big)\approx \frac{H(x-R)}{H\big(h^{-1}(1/t)\big)}\wedge 1,\quad x> R,\, t>0,
	\]
	and 
	\[
	\P^x\big(T_{B_R}>t\big)\approx \frac{|x+R|}{h^{-1}(1/t)}\wedge 1,\quad x<- R,\, t>0.
	\]
\end{corollary}

\subsection{A class of processes which satisfy the assumptions of Theorem \ref{thm:main_hitting_time}}\label{subsec:ex}
Let us now provide an example of a class of non-symmetric L\'{e}vy processes which satisfy the assumptions of Theorem \ref{thm:main_hitting_time}. As one can suspect, the main difficulty here is the lower estimate on $K^{\lambda}$ for small $\lambda$, which is far from obvious for general non-symmetric process, even if the remaining two assumptions are satisfied. Note that if process is symmetric then the third assumption follows from \cite[Lemma 2.15]{GR17}. 

Let $\nu$ be of the form \eqref{gest-miar} such that $\Re\psi \in \WLSC{\alpha}{\chi}$, for some $\alpha>1$ and $\chi \in (0,1]$. Since $\int_{|z|\geq 1}|z|\nu(dz)<\infty$, the characteristic exponent $\psi$ is differentiable  and 
\[
(\Im\psi)'(\xi)=\int_\R z(1-\cos\xi z)\,\nu(dz)=(C_u-C_d)\int^\infty_0 z(1-\cos\xi z)\,\nu_0(dz),\quad \xi\in \R.
\]
Now we specify $\nu_0$. Assume that $0<\beta_1\leq \beta_2<1$ and $0<a_2\leq 1\leq a_1$. Let $\nu_0(dz)=\frac{f(z)}{z^2}dz$, where $f$ is non-negative, non-increasing and satisfies
\[
a_2\lambda^{-\beta_2} f(z)\leq f(\lambda z)\leq a_1\lambda^{-\beta_1}f(z), \quad \lambda>1,\,z>0.
\]
For such $\nu_0$ it is easy to verify that $\Re\psi$ is non-decreasing on $[0,\infty)$ and by \cite[Proposition 28]{BGR14}, there is $c_1=c_1(\beta_1,\beta_2,a_1,a_2)$ such that 
\[
\big|(\Im\psi)'(\xi)\big|\leq c |C_u-C_d| f(1/\xi)\leq c_1\frac{|C_u-C_d|}{C_u+C_d}\frac{\Re\psi(\xi)}{\xi}, \quad \xi\in\R.
\]
Next, we obtain the lower bound for $K^\lambda$ for small $\lambda$. We have, for $x>0$,
\begin{align*}
K^\lambda(x)&=\frac{1}{\pi}\int^\infty_0(1-\cos x\xi)\frac{\lambda+\Re\psi(\xi)}{(\lambda+\Re\psi(\xi))^2+(\Im\psi(\xi))^2}\,d\xi\\ &+\frac{1}{\pi}\int^\infty_0\sin x\xi\frac{\Im\psi(\xi)}{(\lambda+\Re\psi(\xi))^2+(\Im\psi(\xi))^2}\,d\xi\\&=\frac{1}{2}H^\lambda(x)+\mathrm{I}_\lambda(x).
\end{align*}
Since $|\Im\psi(\xi)|\leq c_2\Re\psi(\xi)$, $\xi \in\R$, where $c_2=c_2(\beta_1,a_1)$, by \cite[Lemma 2.15]{GR17}, for $x>0$ and $\lambda\leq h(x)$,
\[
H^\lambda(x)\geq \frac{1}{\pi(1+c_2^2)}\int^\infty_0 (1-\cos x\xi)\frac{d\xi}{\lambda+\Re\psi(\xi)}\geq c_3 \frac{1}{x h(x)},
\]
where $c_3$ depends only on $\beta_1$ and $a_1$. The integration by parts implies, for $x>0$,
\begin{align*}
\pi x\mathrm{I}_\lambda(x)&=\int^\infty_0(1-\cos x\xi)\frac{g(\xi)}{\left((\lambda+\Re\psi(\xi))^2+(\Im\psi(\xi)^2\right)^2}\,d\xi,
\end{align*} 
where 
\begin{align*}
g(\xi)&=2\Im\psi(\xi)\big((\Re\psi)'(\xi)(\lambda+\Re\psi(\xi))+(\Im\psi)'(\xi)\Im\psi(\xi)\big) \\ &-(\Im\psi)'(\xi)\left((\lambda+\Re\psi(\xi))^2+(\Im\psi(\xi)^2\right).
\end{align*}
Assume that $C_u\geq C_d$, then $\Im\psi$, $(\Im\psi)'$, $(\Re\psi)'$ are non-negative on the positive half-line, therefore
\begin{align*}
\pi x\mathrm{I}_\lambda(x)&\geq -\int^\infty_0(1-\cos x\xi)\frac{(\Im\psi)'(\xi)}{(\lambda+\Re\psi(\xi))^2+(\Im\psi(\xi)^2}\,d\xi \\&\geq -c_1\frac{C_u-C_d}{C_u+C_d}\int^\infty_0(1-\cos x\xi)\frac{\Re\psi(\xi)}{\xi\big((\lambda+\Re\psi(\xi))^2+(\Im\psi(\xi)^2\big)}\,d\xi,\\
&\geq -c_1\frac{C_u-C_d}{C_u+C_d}\int^\infty_0(1-\cos x\xi)\frac{1}{\xi\Re\psi(\xi)}\,d\xi \\ &\geq -c_4\frac{C_u-C_d}{C_u+C_d}\frac{1}{\Re\psi(1/x)},
\end{align*} 
where in the last inequality we used \cite[Corollary 22]{BGR14} and $c_4$ depends only on $\beta_1,\beta_2,a_1$ and $a_2$.
Finally we obtain
$$K^\lambda(x)\geq \frac{1}{xh(x)}\left(c_3-\frac{c_4}{\pi}\frac{C_u-C_d}{C_u+C_d}\right),\quad \lambda\leq h(x).$$
Hence, for small $\frac{C_u-C_d}{C_u+C_d}$ we have, for $x>0$, 
$$K^{\lambda}(x)\approx \frac{1}{xh(x)},\quad \lambda\leq h(x).$$
For $x<0$ additional assumption on $f(s)-sf'(s)$ are needed in order to provide similar calculations. 
\bibliographystyle{abbrv}
\bibliography{bibfile}

\end{document}